\documentclass[a4paper]{amsart}
\usepackage{amsmath,amssymb,amsthm,lmodern,mathdots,microtype,tikz-cd}
\usepackage[cmtip,all]{xy}
\usepackage[T1]{fontenc}
\usepackage[colorlinks=true,allcolors=blue!60!black]{hyperref}
\usepackage[capitalise]{cleveref}
\usepackage[left=3.5cm,right=3.5cm,top=3.25cm,bottom=3.25cm]{geometry}
\usepackage{setspace}
\usepackage{bbold}

\let\amsamp=&

\allowdisplaybreaks

\theoremstyle{plain}
\newtheorem{thm}{Theorem}[section]
\newtheorem*{thm*}{Theorem}
\newtheorem{prop}[thm]{Proposition} 
\newtheorem*{prop*}{Proposition} 
\newtheorem{lemma}[thm]{Lemma}
\theoremstyle{definition} 
\newtheorem{defn}[thm]{Definition}
\newtheorem*{note*}{Note}
\newtheorem{rem}[thm]{Remark}
\newtheorem{exmp}[thm]{Example}
\newtheorem{corollary}[thm]{Corollary}

\newcommand{\bgd}[2]{^{{#1},{#2}}} 
\newcommand{\cl}[2]{{#1}\mathrm{\hbox{-}{#2}}} 
\newcommand{\Fun}{\mathrm{Fun}} 
\newcommand{\kk}{R} 
\newcommand{\op}{\mathrm{op}} 
\newcommand{\N}{\mathbb{N}} 
\newcommand{\Z}{\mathbb{Z}} 
\DeclareMathOperator{\im}{im} 

\newcommand{\kkpic}{\vcenter{\hbox{\tiny$\bullet$}}} 
\newcommand{\kkpqpic}{*} 

\newcommand{\Cc}{\mathcal{C}}
\newcommand{\Ee}{\mathcal{E}}

\newcommand{\Mm}{\mathcal{M}}
\newcommand{\Ww}{\mathcal{W}}
\newcommand{\Yy}{\mathcal{Y}}
\newcommand{\Zzw}{\mathcal{ZW}}
\newcommand{\Bbw}{\mathcal{BW}}

\newcommand{\vbic}{\mathrm{vbC}_\kk} 
\newcommand{\ncpx}{n\text{-}\mathrm{mC}_\kk} 
\newcommand{\cpx}[1]{#1\text{-}\mathrm{mC}_\kk} 
\newcommand{\nCh}{n\text{-}\mathrm{mC}} 
\newcommand{\Mod}{\text{-}\mathrm{Mod}} 

\newcommand{\hmat}[2]{\mbox{\tiny$\begin{pmatrix} #1&\hspace{-0.75em}#2\end{pmatrix}$}}
\newcommand{\vmat}[2]{\mbox{\tiny$\begin{pmatrix} #1\\ #2\end{pmatrix}$}}
\newcommand{\idmat}{\mbox{\tiny$\begin{matrix} 1\end{matrix}$}}

\newcommand{\mat}[4]{\mbox{\tiny$\begin{pmatrix} #1&\hspace{-0.75em}#2 \\ #3&\hspace{-0.75em}#4 \end{pmatrix}$}}

\title{Model category structures on multicomplexes}
\author{Xin Fu}
\address[Xin Fu]{Department of Mathematics, Ajou University, San 5, Woncheon-dong, Yeongtonggu, Suwon 443-749, Republic of Korea}
\email{xfu87@ajou.ac.kr}

\author{Ai Guan}
\address[Ai Guan]{Department of Mathematics and Statistics, Lancaster University, Lancaster LA1 4YF, UK}
\email{a.guan@lancaster.ac.uk}

\author{Muriel Livernet}
\address[Muriel Livernet]{Universit\'e de Paris, Sorbonne Universit\'e, CNRS, Institut de Mathématiques de Jussieu-Paris Rive Gauche, F-75013 Paris, France}
\email{livernet@math.univ-paris-diderot.fr}

\author{Sarah Whitehouse}
\address[Sarah Whitehouse]{
School of Mathematics and Statistics\\
University of Sheffield\\ Sheffield\\ S3 7RH\\ UK}
\email{s.whitehouse@sheffield.ac.uk }

\date{\today}



\begin{document}
\maketitle

\begin{abstract}
  We present a family of model structures on the category of multicomplexes. There is a cofibrantly generated model structure in which the weak equivalences are the morphisms inducing an isomorphism at a fixed stage 
  of an associated spectral sequence. Corresponding model structures are given for truncated versions of multicomplexes, interpolating between bicomplexes and multicomplexes. For a fixed stage of the spectral sequence,
  the model structures on all these categories  are shown to be Quillen equivalent.
\end{abstract}

\section{Introduction}

We provide a family of model structures on the category of multicomplexes of $\kk$-modules, for $\kk$ a commutative unital ring. 
A multicomplex is an algebraic structure generalizing the notion of a (graded) chain complex
and that of a bicomplex. The structure involves a family of higher ``differentials'' indexed by the non-negative integers. The terms twisted chain complex and $D_\infty$-module are also used. 
Multicomplexes have arisen in many different places and play an important role in homotopical and homological algebra.
A multicomplex has an associated total complex, with filtration, and thus an associated spectral sequence. 

For each $r\geq 0$, we show that there is a cofibrantly generated model structure on the category of multicomplexes in which the weak equivalences are the morphisms inducing an isomorphism at the
$(r+1)$-th page of the spectral sequence. The fibrations are explicitly specified via surjectivity conditions. 

We also provide such models for certain truncated versions of these structures, the $n$-multicomplexes. 
We write $\ncpx$ for the category of $n$-multicomplexes.
The case $n=2$ gives the category of bicomplexes and the results here recover those of ~\cite{celw}. Multicomplexes can be thought
of as the case $n=\infty$ and we make frequent use of this  notational device.

A key ingredient of the model structures is the explicit description of the spectral sequence associated to a multicomplex in~\cite{lwz}. The main techniques  imitate the work of~\cite{celw}, using representable versions of $r$-cycles and $r$-boundaries to provide 
generating (trivial) cofibrations for the model structures. One difference, however, is that we describe the representing objects for cycles via an iterated pushout process, as a direct description would be cumbersome. 
The model structures appear in Theorems~\ref{modelZ} and~\ref{modelE}, the latter being a minor variant of the former.

We introduce a graded associative algebra $\mathcal C_n$ in the category of vertical bicomplexes such that $n$-multicomplexes can be viewed as $\mathcal C_n$-modules in vertical bicomplexes.
This allows us to set up, for fixed $r$,  Quillen adjunctions relating the model structures of Theorem~\ref{modelE} on $n$-multicomplexes as $n$ varies. Indeed,  the functors can be viewed as restriction and extension of scalars.
We show that these adjunctions are Quillen equivalences for $n\geq 2$. Multicomplexes can be viewed as the \emph{homotopy-coherent version} of bicomplexes~\cite{lv, lrw13}, so that one would expect 
$\cpx{\infty}$ and  $\cpx{2}$ to have equivalent homotopy theories. Our work confirms that this is the case for the $r$-model structure for each $r$ and that
the same is true for all the intermediate categories of $n$-multicomplexes.
\medskip

Our results can be summarized as follows.
We have a chain of adjunctions:
\begin{center}
  \begin{tikzcd}
    \cpx{1} \arrow[r, shift right] & \cpx{2} \arrow[l, shift right] \arrow[r, shift right] & \cpx{3} \arrow[l, shift right] \arrow[r, shift right] & \cdots \arrow[l, shift right] \arrow[r, shift right] & \ncpx \arrow[l, shift right] \arrow[r, shift right] & \cdots \arrow[l, shift right] \arrow[r, shift right] & \cpx{\infty} \arrow[l, shift right].
  \end{tikzcd}
\end{center}
Apart from at the far left, we may fix any $r\geq 0$ and endow the categories with the $r$-model structure of Theorem~\ref{modelE}. Equipped with these model structures, each adjoint pair, apart from the leftmost one, gives a Quillen equivalence.
The category
$\cpx{1}$ is the category of vertical bicomplexes, where the objects have only one non-trivial structure map, a vertical differential. In this case, we only have the $r=0$ model structure, corresponding to the usual projective model structure on cochain complexes. Indeed, in this case the associated spectral sequence degenerates at the $E_1$ page and the notions of equivalence in our hierarchy all coincide. The leftmost adjoint pair gives a Quillen adjunction for $r=0$, but it is not a Quillen equivalence.

In the category of $n$-multicomplexes with the $r$-model structure of Theorem~\ref{modelE}, let the weak equivalences be denoted
$\mathcal{E}_r^n$, the fibrations $\it{Fib}_r^n$ and the cofibrations
$\it{Cof}_r^n$. Then for all $n\geq 2$ and $r\geq 0$ we have
\[
  \mathcal{E}_r^n\subseteq \mathcal{E}_{r+1}^n, \qquad  {\it{Fib}}_{r+1}^n\subseteq {\it{Fib}}_r^n, \qquad
  {\it{Fib}}_r^n\cap  \mathcal{E}_r^n \subseteq {\it{Fib}}_{r+1}^n\cap \mathcal{E}_{r+1}^n ,\qquad {\it{Cof}}_{r+1}^n\subseteq {\it{Cof}}_r^n.  
\]

\medskip

The paper is arranged as follows.
In \cref{S:preliminaries} we give the necessary background on multicomplexes and related categories.
\cref{S:modcat} presents the $r$-model category structure on these categories for each $r\geq 0$. 
In \cref{S:relationships} we describe the relationships between these model structures, setting up the Quillen adjunctions and equivalences.
\cref{S:bounded} considers analogues of the previously obtained model category structures for bounded multicomplexes.
\cref{S:cofibrants} gives various examples of cofibrations and cofibrant replacements for our model category structures, in both the bounded and unbounded cases.

\subsection*{Acknowledgements} We would like to thank the
Hausdorff Research Institute for Mathematics for hosting the Women in Topology III workshop and for financial support. We would also like to thank the NSF (NSF-DMS 1901795, NSF-HRD 1500481 -- AWM ADVANCE grant) and
Foundation Compositio Mathematica for financial support for this event.

\section{Notations and preliminaries}
\label{S:preliminaries}

In this section, we summarize the definitions and results on multicomplexes that are required for the rest of the paper, and fix sign and grading conventions.
Throughout this paper, $\kk$ will denote a commutative unital ground ring and $\kk\Mod$ will denote the category of $\kk$-modules.

\begin{defn}
  A \emph{multicomplex} or \emph{$\infty$-multicomplex} $A$ is a $(\Z,\Z)$-bigraded $\kk$-module $A = \{A^{p,q}\}_{p,q\in\Z}$  endowed with a family of maps $\{d_i \colon A \to A\}_{i \ge 0}$ of bidegree $(-i,1-i)$ satisfying for all $l \ge 0$,
  \begin{equation}\label{E:multicomplex}
    \sum_{i+j=l} (-1)^i d_id_j=0.
  \end{equation}
  Let $n \ge 1$ be an integer.
  An \emph{$n$-multicomplex} is a multicomplex with $d_i=0$ for all $i\geq n$.

  For $1 \leq n \leq \infty$, a \emph{\textup(strict\textup) morphism} of $n$-multicomplexes is a map $f$ of bigraded $\kk$-modules of bidegree $(0,0)$ satisfying $d_if=fd_i$ for all $i\geq 0$.
  We denote by $\ncpx$ the category of $n$-multicomplexes and strict morphisms.
\end{defn}

For example, a $1$-multicomplex is a vertical bicomplex, as defined in \cite[Section 2.1]{lrw13}, that is, a $(\Z,\Z)$-bigraded $\kk$-module endowed with a differential $d_0$ of bidegree $(0,1)$. 
The category $\cpx{1}$ will also be denoted by $\vbic$ when we need to emphasize vertical bicomplexes.

A $2$-multicomplex is a bicomplex, with the convention $d_0d_1=d_1d_0$; thus the chosen sign convention agrees with~\cite[Definition 2.10]{celw}. 

As observed in~\cite[Remark 2.2]{lwz}, the above choice of sign convention for multicomplexes gives an isomorphic category to the version without signs in the relations.

\medskip

By~\cite[Lemma 3.3]{celw18}, the category of multicomplexes is symmetric monoidal, where the monoidal structure is given by the bifunctor 
\[\otimes \colon \cpx{\infty}\times \cpx{\infty}\to \cpx{\infty}\]
which on objects is given by
$((A,d_i^A),(B,d_i^B))\mapsto (A\otimes B, d^A_i\otimes 1 + 1\otimes d^B_i)$
and on strict morphisms is given by
$(f, g)\mapsto f\otimes g$.
The symmetry isomorphism is given by the morphism of multicomplexes
\[\tau_{A\otimes B} \colon A\otimes B \to B\otimes A\]
defined by
\[a\otimes b \mapsto (-1)^{\langle a, b \rangle}b\otimes a.\]
Here for $a$, $b$ of bidegree $(a_1,a_2)$, $(b_1,b_2)$ respectively, we let $\langle a,b \rangle=a_1b_1+a_2b_2$.

This functor also describes a symmetric monoidal structure on $\ncpx$ for each $n\geq 1$ by restriction.

For the rest of this section, let $r \geq 0$ be an integer.
We consider the spectral sequence $E_r^{*,*}(A)$ associated to the multicomplex $A$ as described in \cite[Proposition 2.8]{lwz}. 
The following is a reformulation of the description in~\cite[Definition 2.6]{lwz} to make the notation consistent with~\cite{celw} in the case of bicomplexes.

\begin{prop}[cf.~\protect{\cite[Lemma 2.13]{celw}}]
  Let\/ $(A,d_0,d_1,\dots,d_n,\dots)$ be a multicomplex. Then
  \[
    E_r^{p,q}(A) \cong Z_r^{p,q}(A) / B_r^{p,q}(A)
  \]
  where the cycles are $Z_0^{p,q}(A) := A^{p,q}$ and for all $r \ge 1$,  
  \begin{align*}  
    Z_r^{p,q}(A) := \bigl\{a_0 \in A^{p,q} \mid{} & \text{for all\/ $0 \le l \le r-1$,}\\
    &\text{$\textstyle\sum_{i+j=l} (-1)^id_ia_j = 0$ for some $a_j \in A^{p-j,q-j}$, $1 \le j \le r-1$} \bigr\},
  \end{align*}
  and the boundaries are $B_0^{p,q}(A) := 0$, $B_1^{p,q} (A) = A^{p,q} \cap \im{d_0}$,
  and for all $r \ge 2$,
  \begin{equation}\label{br}
    \begin{split}    
      B_r^{p,q}(A) := \bigl\{x \in A^{p,q} \mid{} & \text{there exist $b_i \in A^{p+r-1-i,q+r-2-i}$ for $0 \le i \le r-1$ such that}\\
        & x=\sum_{i=0}^{r-1}(-1)^i d_ib_{r-i-1},\\
      &\text{and $\sum_{i=0}^{l}(-1)^i d_ib_{l-i}=0$ for $0 \le l \le r-2$}\bigr\}.
    \end{split}
  \end{equation}

  The differential $\Delta_r \colon Z_r^{p,q}(A)/B_r^{p,q}(A)\rightarrow  Z_r^{p-r,q+1-r}(A)/B_r^{p-r,q+1-r}(A)$
  is given by 
  \[
    \Delta_r([a_0])=\left[\sum_{i=1}^r (-1)^i d_i a_{r-i}\right]. \eqno\qed
  \]
\end{prop}

\begin{defn}
  Let $2\leq n\leq \infty$. 
  A morphism of $n$-multicomplexes $f \colon A\to B$ is said to be an \textit{$E_r$-quasi-isomorphism} if
  the morphism $E_r(f) \colon E_r(A)\to E_r(B)$ at the $r$-stage of the associated spectral 
  sequence is a quasi-isomorphism of $r$-bigraded complexes (that is, $E_{r+1}(f)$ is an isomorphism).
\end{defn}

Denote by $\Ee_r^n$ the class of $E_r$-quasi-isomorphisms of $\cpx{n}$. 
This class contains all isomorphisms of $\cpx{n}$, satisfies the two-out-of-three property and is closed under retracts.

\medskip
Finally, we recall from \cite{celw18} the definition of $r$-homotopies of multicomplexes in the context of strict morphisms.

\begin{defn}\cite[Proposition 3.18]{celw18} \label{D:rhomotopy}
  Let $f,g \colon A \to B$ be two strict morphisms of multicomplexes. 
  An \emph{$r$-homotopy} $h$ from $f$ to $g$ is a collection of maps $h_m \colon A \to B$ of bidegree $(-m+r, -m+r-1)$ satisfying for all $m\geq 0$,
  \[
    \sum_{i+j=m} (-1)^{i+r}d_ih_j+(-1)^i h_id_j=\begin{cases} g-f& \text{ if } m=r,\\ 0& \text{ if } m\not=r.\end{cases}
  \]
  We write $f \simeq_r g$ if there is an $r$-homotopy from $f$ to $g$.
  A morphism $f \colon A \to B$ is an \emph{$r$-homotopy equivalence} if there exists a morphism $g \colon B \to A$ such that $f \circ g \simeq_r 1_B$ and $g \circ f \simeq_r 1_A$.
  A multicomplex $A$ is \emph{$r$-contractible} if $1_A \simeq_r 0$.
\end{defn}

Any $r$-homotopy equivalence is an $E_r$-quasi-isomorphism by \cite[Proposition 3.24]{celw18}.

\section{Model structures on multicomplexes and $n$-multicomplexes, for $n\geq 2$}
\label{S:modcat}

We now describe our model category structures on $n$-multicomplexes, for $2 \le n \le \infty$. 
In the case $n = 2$, the model category structures here are precisely those of bicomplexes obtained in \cite{celw}, and indeed, the proofs for general $n$-multicomplexes are essentially the same. 
Just like for bicomplexes, a key idea in the proof is to show that the spectral sequence admits a description in terms of certain \emph{witness} functors that have the advantage of being representable.
Our presentation here differs from \cite[Sections 4.1--4.2]{celw} in that we show the representing objects for the witness functors can be defined recursively; this is helpful for avoiding notational difficulties in the general multicomplex case.

\subsection{Cofibrantly generated model categories}

We collect some definitions and results on cofibrantly generated model categories from~\cite{Hovey}.

\begin{defn}
  Let $\Cc$ be a category with all small colimits and limits and $I$ be a class of maps in $\Cc$.
  \begin{enumerate}
    \item A morphism is called \textit{$I$-injective} (resp.~\textit{$I$-projective}) if it has the right (resp.~left)
      lifting property with respect to morphisms in $I$. We write
      \[\cl{I}{inj}:=\mathrm{RLP}(I)\text{ and }\cl{I}{proj}:=\mathrm{LLP}(I).\]
    \item  A morphism is called an \textit{$I$-fibration} (resp.~\textit{$I$-cofibration}) if it has the right (resp.~left) lifting property with respect to 
      $I$-projective (resp.~$I$-injective) morphisms. We write
      \[\cl{I}{fib}:=\mathrm{RLP}(\cl{I}{proj})\text{ and }\cl{I}{cof}:=\mathrm{LLP}(\cl{I}{inj}).\]
    \item A map is a \textit{relative $I$-cell complex} if it is a transfinite composition of pushouts of
      elements of $I$. We denote by $\cl{I}{cell}$ the class of relative $I$-cell complexes.
  \end{enumerate}
\end{defn}

\begin{defn}A model category $\Cc$ is said to be  \textit{cofibrantly
  generated} if there are sets $I$ and $J$ of maps such that the following conditions hold.
  \begin{enumerate}
    \item The domains of the maps of $I$ are small relative to $\cl{I}{cell}$.
    \item  The domains of the maps of $J$ are small relative to $\cl{J}{cell}$.
    \item  The class of fibrations is $J$-inj.
    \item The class of trivial fibrations is $I$-inj.
  \end{enumerate}
  The set $I$ is called the \textit{set of generating cofibrations}, and $J$ the \textit{set of generating trivial cofibrations}.
\end{defn}

The following is a consequence of Kan's Theorem 
(cf. \cite[Theorem 11.3.1]{Hir} or \cite[Theorem 2.1.19]{Hovey}).

\begin{thm}[D.~M.~Kan]
  Suppose\/ $\Cc$ is a category with all small colimits and limits.
  Let $\Ww$ be a subcategory of\/ $\Cc$ and $I$ and $J$ be sets of maps in\/ $\Cc$.
  Then there
  is a cofibrantly generated model structure on\/ $\Cc$ with $I$ as the set of generating
  cofibrations, $J$ as the set of generating trivial cofibrations, and\/ $\Ww$ as the subcategory
  of weak equivalences if and only if the following conditions are satisfied.
  \begin{enumerate}
    \item\label{T:Hovey1} The subcategory\/ $\Ww$ satisfies the two-out-of-three property and is closed under
      retracts.
    \item\label{T:Hovey2} The domains of\/ $I$ are small relative to\/ $\cl{I}{cell}$.
    \item\label{T:Hovey3} The domains of\/ $J$ are small relative to\/ $\cl{J}{cell}$.
    \item $\cl{J}{cof}\subseteq \Ww$.
    \item $\cl{I}{inj}= \Ww\cap \cl{J}{inj}$.
  \end{enumerate}
\end{thm}

Note that the categories of $n$-multicomplexes we will consider satisfy the assumptions of this theorem as well as conditions (\ref{T:Hovey1}), (\ref{T:Hovey2}) and (\ref{T:Hovey3}).

\subsection{Witness cycles and witness boundaries in multicomplexes}

We begin by defining the witness cycles and witness boundaries functors and showing that they can be used to describe the spectral sequence of a multicomplex.

\begin{defn}
  Let $A$ be a multicomplex and $r \ge 0$.

  Define the \emph{witness $r$-cycles} $ZW_r^{p,q}(A)$ to be the bigraded $\kk$-modules 
  $ZW_0^{p,q}(A) = A^{p,q}$
  and for $r\ge 1$,
  \begin{align*}
    ZW_r^{p,q}(A)=\{(a_0, a_1,\ldots, a_{r-1}) \mid &~a_i \in A^{p-i, q-i}~\text{for}~0\leq i\leq r-1~\text{such~that}\\
      &~\sum_{i+j=l} (-1)^id_ia_j = 0~\text{for}~0\leq l\leq r-1
    \}.
  \end{align*}
  There is a natural map of bigraded $\kk$-modules
  \[
    z_r \colon ZW_r^{p,q}(A) \to Z_r^{p,q}(A)
  \] 
  given by $z_0=1_A$, and for $r\ge 1$,
  \[
    z_r(a_0,\dots,a_{r-1}) = a_0.
  \]

  Define the \emph{witness $r$-boundaries} to be the bigraded $\kk$-modules $BW_0^{p,q}(A) = 0$, $BW_1^{p,q}(A) = A^{p,q}$ and for $r \ge 2$,
  \[
    BW_r^{p,q-1}(A)=ZW_{r-1}^{p+r-1,q+r-2}(A) \oplus A^{p,q-1} \oplus ZW_{r-1}^{p-1,q-1}(A).
  \]
  Writing elements of $BW_r^{*,*}(A)$ as $(b_0,\dots,b_{r-2};a;c_0,\dots,c_{r-2})$ with $a \in A^{*,*}$ and $(b_0,\dots,b_{r-2})$, $(c_0,\dots,c_{r-2}) \in ZW_{r-1}^{*,*}(A)$, there is a natural bidegree $(0,1)$ map of bigraded $\kk$-modules
  \[
    \beta_r \colon BW_r^{p,q-1}(A) \to B_r^{p,q}(A)
  \]
  given by $\beta_0=0$, $\beta_1=d_0$ and for $r \ge 2$, 
  \[
    (b_0,\dots,b_{r-2}; a; c_0,\dots,c_{r-2}) \longmapsto
    d_0a+\sum_{i=1}^{r-1} (-1)^i d_ib_{r-i-1}.
  \]
\end{defn}

We note that the maps $z_r$ and $\beta_r$ are surjective.
\smallskip

The final ingredient we need here is a map from witness boundaries to witness cycles. 
The following lemma is a check necessary for the definition of this map.

\begin{lemma}\label{checkforw} For $r\geq 2$,
  the map  of bigraded $\kk$-modules specified by
  \[
    (b_0,\ldots,b_{r-2}) \longmapsto \big(\sum_{i=1}^{r-1} (-1)^i d_ib_{r-1-i}, -\sum_{i=2}^{r} (-1)^i d_ib_{r-i},  \dots, 
    (-1)^{r-1}\sum_{i=r}^{2r-2}(-1)^i d_ib_{2r-2-i} \big),
  \]
  gives a map from 
  $ZW_{r-1}^{p+r-1,q+r-2}(A)$
  to $ZW_r^{p,q}(A)$.
\end{lemma}

\begin{proof}
  Let $ \underline{b}=(b_0,\ldots,b_{r-2})\in ZW_{r-1}^{p+r-1,q+r-2}(A)$ and for $0 \leq j \leq r-1$ let 
  \[
    \alpha_j=(-1)^j\sum_{k=j+1}^{r+j-1} (-1)^{k} d_kb_{r+j-1-k}.
  \]
  Proving that $(\alpha_0,\dots,\alpha_{r-1})\in ZW_r^{p,q}(A)$ amounts to computing, for $0\leq l\leq r-1$:
  \begin{align*}
    \sum_{i+j=l}(-1)^id_i\alpha_j
    &=\sum_{t=0}^{r-2}(-1)^{r-1-t}\left(\sum_{i=0}^l (-1)^i d_id_{r+l-1-t-i}\right) b_t\\
    &=\sum_{t=0}^{r-2}(-1)^{r-t}\left(\sum_{i=l+1}^{r+l-1-t} (-1)^i d_id_{r+l-1-t-i}\right) b_t\\
    &\qquad\qquad\text{by the multicomplex relations}\\
    &=\sum_{i=l+1}^{l+r-1}(-1)^{l-1}d_i\left(\sum_{t=0}^{r+l-1-i}(-1)^{l+r-1-t-i}d_{r+l-1-t-i}b_t\right)\\
    &=0 \qquad\text{since $\underline{b}\in ZW_{r-1}^{p+r-1,q+r-2}(A)$}.
  \end{align*}
  Thus the image of $(b_0,\ldots,b_{r-2})$ lies in  $ZW_r^{p,q}(A)$, as required.
\end{proof}

\begin{defn}
  The bidegree $(0,1)$ map of bigraded $\kk$-modules
  \[
    w_r \colon BW_r^{p,q-1}(A) \to ZW_r^{p,q}(A)
  \] 
  is given by $w_0=0$, $w_1=d_0$ and for $r\ge 2$,
  \begin{multline*}
    (b_0,\dots,b_{r-2}; a; c_0,\dots,c_{r-2})\\
    \overset{w_r}\longmapsto \big(d_0a+\sum_{i=1}^{r-1} (-1)^i d_ib_{r-1-i}, d_1a-\sum_{i=2}^{r} (-1)^i d_ib_{r-i}+c_0, d_2a+ \sum_{i=3}^{r+1} (-1)^i d_ib_{r+1-i}+c_1, \dots, \\ 
    d_{r-1}a+(-1)^{r-1}\sum_{i=r}^{2r-2} (-1)^id_ib_{2r-2-i} +c_{r-2} \big).
  \end{multline*}

  This is well-defined as it is the sum of the map defined in Lemma~\ref{checkforw} and the maps
  \[
    A^{p,q-1} \longrightarrow ZW_r^{p,q}(A),\ 
    a \longmapsto (d_0a,d_1a,d_2a,\ldots,d_{r-1}a)\\
  \]
  and
  \[
    ZW_{r-1}^{p-1,q-1}(A)\longrightarrow ZW_r^{p,q}(A),\ 
    (c_0,\ldots,c_{r-2}) \longmapsto (0,c_0,\ldots,c_{r-2})
  \]
  which are well-defined due to the definition of multicomplexes and of $ZW_{r}$.
\end{defn}

All these definitions extend naturally to functors from multicomplexes to $\kk$-modules and natural transformations. 
By abuse of notation we will also denote by $ZW^{p,q}_r$, $BW^{p,q}_r$ the restriction of these functors to the category of $n$-multicomplexes.

\begin{prop}[cf.~\protect{\cite[Proposition 4.3]{celw}}]
  For every $r \ge 0$, for every $p,q\in\Z$, and for $2\leq n\leq \infty$ there is a commutative diagram of natural transformations of functors $\ncpx\to\kk\Mod$
  \begin{center}
    \begin{tikzcd}
      BW_r^{p,q-1} \arrow[r, "w_r"] \arrow[d, "\beta_r"'] & ZW^{p,q}_r \arrow[d, "z_r"] & \\
      B^{p,q}_r \arrow[r, hook,"\iota_r"] & Z^{p,q}_r \arrow[r, two heads,"\pi"] & E^{p,q}_r
    \end{tikzcd}
  \end{center}
  and the natural transformation $\pi_r=\pi\circ z_r \colon ZW^{p,q}_r \to E^{p,q}_r$ induced by the above diagram satisfies
  \[
    \ker \pi_r(A) = \im w_r(A)
  \]
  for every $n$-multicomplex $A$. In particular, $E_r^{p,q}(A)\cong ZW_r^{p,q}(A)/w_r(BW^{p,q-1}_r(A))$.

  Under this isomorphism, the differential on the $r$-page of the spectral sequence
  \[
    \delta_r \colon ZW_r^{p,q}(A)/w_r(BW^{p,q-1}_r(A))\to ZW_r^{p-r,q+1-r}(A)/w_r(BW^{p-r,q-r}_r(A))
  \]
  is given by
  \[
    [(a_0, a_1,\dots, a_{r-1})]\mapsto \left[(\sum_{i=1}^r(-1)^i d_ia_{r-i}, \sum_{i=1}^r(-1)^id_{i+1}a_{r-i},
    \dots, \sum_{i=1}^r(-1)^i d_{i+r-1} a_{r-i})\right].
  \]
\end{prop}

\begin{proof}
The result is trivial for $r=0$ and $r=1$, so we consider $r\geq 2$.
  It is straightforward to check that the diagram commutes. 
  We next show that $\ker \pi_r\subseteq \mathrm{im}\,w_r$.
  If $\underline{a}=(a_0, a_1,\ldots, a_{r-1})\in ZW^{p,q}_r(A)$ satisfies $\pi_r(\underline{a})=0$, this means that $z_r(\underline{a})\in B^{p,q}_r(A)$, i.e., $a_0\in B^{p,q}_r(A)$. By~\eqref{br}, there exists 
  $(b_0, b_1, \ldots, b_{r-2}; b_{r-1})\in ZW_{r-1}^{p+r-1,q+r-2}(A)\oplus A^{p,q-1}$
  such that $a_0=\sum_{i=0}^{r-1}(-1)^i d_ib_{r-i-1}$.

  Let us compute
  \begin{multline*}
    (a_0,\ldots,a_{r-1})-w_r(b_0,\ldots,b_{r-2};b_{r-1};0,\ldots,0)=\\
    (0,\underbrace{a_1+\sum_{i=1}^{r} (-1)^i d_ib_{r-i}, a_2-\sum_{i=2}^{r+1} (-1)^i d_ib_{r+1-i}, \dots,a_{r-1}-(-1)^{r-1}\sum_{i=r-1}^{2r-2} (-1)^id_ib_{2r-2-i}}_{=:(0,c_0,c_1,\ldots,c_{r-2})}).
  \end{multline*}

  A computation shows that $(c_0,c_1,\dots,c_{r-2})\in ZW^{p-1,q-1}_{r-1}(A)$ so that 
  \[
    \underline{a}=w_r(b_0, \dots, b_{r-2};b_{r-1};c_0, \dots, c_{r-2})\quad \in  \mathrm{im}\,w_r. 
  \]
  Conversely we have 
  \[
    \pi_r \circ w_r=\pi\circ z_r\circ w_r=\pi\circ\iota_r\circ \beta_r=0,
  \]
  so that $\mathrm{im}\,w_r\subseteq \ker \pi_r$.

  Another calculation shows that the claimed  differential $\delta_r$ gives a well-defined map on 
  $ZW_r(A)/w_r(BW_r(A))$. Indeed, if $\underline{a}=w_r(b_0, \dots, b_{r-2};b_{r-1};c_0, \dots, c_{r-2})$,
  then 
  \[
    \delta_r(\underline{a})=w_r(c_0, \dots, c_{r-2};\beta_{r-1};\gamma_0, \dots, \gamma_{r-2}),
  \]
  where
  $\beta_{r-1}=\sum_{l=0}^{r-1} (-1)^{r+l} d_{2r-1-l}b_l$ and
  $\gamma_j=\sum_{i=0}^j\sum_{k=1}^r (-1)^{k+1} d_id_{r+j+k-i}b_{r-k}$, for $0\leq j\leq r-2$.

  It is straightforward to check that $\delta_r$ corresponds to the differential $\Delta_r$ under the isomorphism.
\end{proof}

\begin{lemma}\label{L:kerwr}
  Let $A\in \ncpx$.
  For $r\geq 1$, the kernel of the map $w_r \colon BW_r^{p,q-1}(A) \to ZW_r^{p,q}(A)$ is isomorphic to $ZW_r^{p+r-1,q+r-2}(A)$, via the map $(\underline{b};a;\underline{c})\mapsto (\underline{b},a)$.
\end{lemma}

\begin{proof} This is clear from the definition of $w_r$: the element  $(\underline{b};a;\underline{c})$ being in 
  $\ker w_r$ means that $\underline{c}$ is completely determined in terms of $(\underline{b},a)$ and
  that $d_0a=-\sum_{i=1}^{r-1}(-1)^id_ib_{r-1-i}$. 
  Together with $\underline{b}\in ZW_{r-1}^{p+r-1,q+r-2}(A)$, this gives exactly that  $(\underline{b},a)\in  ZW_r^{p+r-1,q+r-2}(A)$.
\end{proof}

The following result is straightforward.

\begin{lemma}\label{L:pullback}
  The following commutative diagrams are pullback squares in the category of $\kk$-modules for every $r\geq 2$ and
  every $n$-multicomplex $A$.
  \begin{center}
    \begin{tikzcd}
      ZW^{p,q}_{1}(A) \arrow[r, hook] \arrow[d] & ZW^{p,q}_0(A) \arrow[d,"d_0"] \\
      0 \arrow[r] & ZW^{p,q+1}_0(A) 
    \end{tikzcd}\qquad
    \begin{tikzcd}
      ZW^{p,q}_{r}(A) \arrow[r,"\pi_{r}"] \arrow[d,"\rho_{r}"'] & ZW^{p-r+1,q-r+1}_0(A) \arrow[d,"d_0"] \\
      ZW^{p,q}_{r-1}(A) \arrow[r,"D_{r-1}"] & ZW^{p-r+1,q-r+2}_0(A) 
    \end{tikzcd}
  \end{center}
  Here $\pi_{r}$ is the projection onto the last coordinate, $\rho_r$ is the projection onto the first $r-1$ components, and
  \[
    D_{r-1} = \sum_{i=1}^{r-1}(-1)^{i+1}d_i \colon (a_0,\ldots,a_{r-2}) \longmapsto \sum_{i=1}^{r-1}(-1)^{i+1}d_ia_{r-1-i}. \eqno\qed
  \]
\end{lemma}

The maps $\pi_r$, $\rho_{r}$, $d_0$ and $D_{r-1}$ define natural transformations between the functors $ZW_r^{p,q}$, and as a consequence we obtain the following proposition.

\begin{prop}\label{P:pullbackfunctors}
  The following commutative diagrams are pullback squares in the functor category $\Fun(\ncpx,\kk\Mod)$ for every $r \ge 2$.
  \begin{center}
    \begin{tikzcd}
      ZW^{p,q}_{1} \arrow[r, hook] \arrow[d] & ZW^{p,q}_0 \arrow[d,"d_0"] \\
      0 \arrow[r] & ZW^{p,q+1}_0
    \end{tikzcd}\qquad
    \begin{tikzcd}
      ZW^{p,q}_{r} \arrow[r,"\pi_{r}"] \arrow[d,"\rho_{r}"' ] & ZW^{p-r+1,q-r+1}_0 \arrow[d,"d_0"] \\
      ZW^{p,q}_{r-1} \arrow[r,"D_{r-1}"] & ZW^{p-r+1,q-r+2}_0 
    \end{tikzcd}
  \end{center} 
\end{prop}

\begin{proof}
  Since $\kk\Mod$ is complete, limits in the functor category exist and they are computed objectwise, so the result follows directly from Lemma~\ref{L:pullback}.
\end{proof}

\begin{rem}\label{mainobs} 
  Similarly to~\cite[Remark 4.5]{celw}, for $2\leq n\leq \infty$, if $f \colon A \to B$ is a morphism of $n$-multicomplexes and $r\geq 1$, then the following are equivalent.
  \begin{enumerate}
    \item The maps $ZW_r(f)$, $ZW_{r-1}(f)$ and $f$ are surjective.
    \item The maps $E_{r}(f)$  and $ZW_{r-1}(f)$ and $f$ are surjective.
  \end{enumerate}
\end{rem}

\subsection{Representing elements} 

We now describe suitable representing objects for the witness cycles and boundaries previously defined.

\begin{defn}
  Let $2 \le n \le \infty$.
  The \emph{$n$-disk at place $(p,q)$}, denoted $\mathbb{D}^n(p,q)$, is the $n$-multicomplex freely generated by a single element $x$ in bidegree $(p,q)$, in the sense of satisfying the following universal construction.
  For any $n$-multicomplex $A$, every map of bigraded sets $\{x\} \to A$ extends uniquely to an $n$-multicomplex morphism $\mathbb{D}^n(p,q) \to A$ such that the following diagram commutes:
  \begin{center}
    \begin{tikzcd}[sep=scriptsize]
      \{x\} \arrow[r, hook] \arrow[rd] & {\mathbb{D}^n(p,q)} \arrow[d] \\
      & A        
    \end{tikzcd}
  \end{center}
\end{defn}

By definition the $n$-multicomplex $\mathbb{D}^n(p,q)$ freely generated by $x$ in bidegree $(p,q)$ is
the quotient of the free bigraded $\kk$-module generated by all finite words $d_{i_1}d_{i_2} \dots d_{i_k} (x)$, $k \geq 0$, $0 \le i_1,\dots,i_k \le n-1$, by the relations
\[
  \sum_{i+j=l} (-1)^i d_id_j(x)=0 \text{ for } l \ge 0,
\]
with differential 
\[
  d_i (d_{i_1}d_{i_2} \dots d_{i_k} (x)) = d_id_{i_1}d_{i_2} \dots d_{i_k} (x).
\]

\begin{rem}
  Using the relations, one can rewrite any word $d_{i_1}d_{i_2} \dots d_{i_k} (x)$ by swapping any occurrence of $d_0$ with all the higher structure maps to its left. 
  The rewriting process and the relation $d_0^2 = 0$ ensure that every word is a linear combination of words of the form
  \begin{equation}\label{E:words}
    d_0^id_{i_1}d_{i_2} \dots d_{i_k}(x), \text{ for } i\in\{0,1\},\ k \ge 0,\ 0 < i_1,\dots,i_k \le n-1.
  \end{equation}
  It is clear from this description that the $d_0$-homology of an $n$-disk is $0$.

  The words listed in \eqref{E:words} above form a basis for the $\infty$-disk; see \cite[Definition 5.4]{mr19} for an explicit description of the $\infty$-disk for multicomplexes concentrated in the right half-plane.
  For $n$ finite, the words listed in \eqref{E:words} are not necessarily distinct or nonzero, so do not form a basis for the $n$-disk.

\end{rem}

\begin{exmp}
  The $3$-multicomplex $\mathbb{D}^3(p,q)$ can be depicted as follows.
  
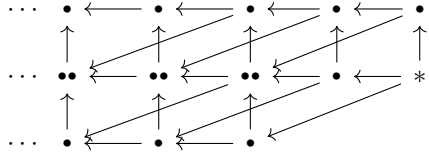
\begin{center}  
    \begin{tikzcd}[sep=scriptsize]
      \cdots\hspace{-2em} & \kkpic            & \kkpic \arrow[l]            & \kkpic \arrow[l] \arrow[lld]            & \kkpic \arrow[l] \arrow[lld]           & \kkpic \arrow[l] \arrow[lld]                     \\
      \cdots\hspace{-2em} & \kkpic\kkpic \arrow[u] & \kkpic\kkpic \arrow[u] \arrow[l] & \kkpic\kkpic \arrow[l] \arrow[u] \arrow[lld] & \kkpic \arrow[u] \arrow[l] \arrow[lld] & \kkpqpic \arrow[u] \arrow[l] \arrow[lld] \\
      \cdots\hspace{-2em} & \kkpic \arrow[u]  & \kkpic \arrow[u] \arrow[l]  & \kkpic \arrow[u] \arrow[l]              &                                   &                                            
    \end{tikzcd}
  \end{center}
  Here each vertex marked $\kkpic$ represents the ring $\kk$, each vertex marked $\kkpic\kkpic$ represents $\kk \oplus \kk$, the vertex marked $\kkpqpic$ represents $\kk$ in bidegree $(p,q)$ and the arrows are
  \begin{align*}
    &d_0 \colon 
    \quad \kkpic \overset{1}\longrightarrow \kkpic,
    \quad \kkpic \xrightarrow{\vmat{0}{1}} \kkpic\kkpic, 
    \quad \kkpic\kkpic \xrightarrow{\hmat{1}{0}} \kkpic, \\ 
    &d_1 \colon 
    \quad \kkpic \overset{1}\longrightarrow \kkpic,
    \quad \kkpic \xrightarrow{\vmat{1}{1}} \kkpic\kkpic, 
    \quad \kkpic\kkpic \xrightarrow{\ \mat{1}{0}{0}{1}\ } \kkpic\kkpic, \\ 
    &d_2 \colon 
    \quad \kkpic \overset{1}\longrightarrow \kkpic,
    \quad \kkpic \xrightarrow{\vmat{1}{0}} \kkpic\kkpic, 
    \quad \kkpic\kkpic \xrightarrow{\hmat{0}{1}} \kkpic. 
  \end{align*}
 
\end{exmp}

\begin{defn}\label{D:curlyZW}
  Let $2\leq n\leq\infty$.
  Define the $n$-multicomplex $\Zzw^n_0(p,q) = \mathbb{D}^n(p,q)$, define $\Zzw^n_1(p,q)$ to be the pushout
  \begin{center}
    \begin{tikzcd}
      \Zzw^n_0(p,q+1) \arrow[r,"d_0^*"] \arrow[d] & \Zzw^n_0(p,q) \arrow[d,"j_0"] \\
      0 \arrow[r] & \Zzw^n_{1}(p,q)
    \end{tikzcd}
  \end{center}
  in the category of $n$-multicomplexes, and for $r \geq 2$, define $\Zzw^n_r(p,q)$ recursively to be the pushout 
  \begin{center}
    \begin{tikzcd}
      \Zzw^n_0(p-r+1,q-r+2) \arrow[r, "d_0^*"] \arrow[d,"D_{r-1}^*"'] & \Zzw^n_0(p-r+1,q-r+1) \arrow[d,"j_{r-1}"] \\
      \Zzw^n_{r-1}(p,q) \arrow[r,"i_{r-1}"] & \Zzw^n_{r}(p,q)
    \end{tikzcd}
  \end{center}
  in the category of $n$-multicomplexes. 
  Here, for all $r \geq 1$, writing $x$ and $a_{r-1}$ for the generators of $\Zzw^n_0(p-r+1,q-r+2)$ and $\Zzw^n_0(p-r+1,q-r+1)$ respectively, the morphism $d_0^*$ is 
  \[
    d_0^*(x) = d_0a_{r-1}.
  \]
 
  By abuse of notation, we also denote the element $j_0(a_0)$ in $\Zzw^n_{1}(p,q)$ by $a_0$.
  For $r \geq 2$, we recursively define the morphism $D_{r-1}^*$ to be
  \[
    D_{r-1}^*(x) = \sum_{i=1}^{r-1}(-1)^{i+1}d_ia_{r-1-i},
  \]
  and again by abuse of notation, we denote the elements $i_{r-1}(a_s)$ ($0 \leq s \leq r-2$) and $j_{r-1}(a_{r-1})$ in $\Zzw^n_{r}(p,q)$ by $a_s$ ($0 \leq s \leq r-2$) and $a_{r-1}$ respectively.
\end{defn}

\begin{exmp}
  The $3$-multicomplex $\Zzw^3_1(p,q)$ can be depicted as:
  \begin{center}
    \begin{tikzcd}[sep=scriptsize]
      \cdots\hspace{-2em} & \kkpic           & \kkpic \arrow[l]           & \kkpic \arrow[l] \arrow[lld] & \kkpic \arrow[l] \arrow[lld] & \kkpqpic \arrow[l] \arrow[lld] \\
      \cdots\hspace{-2em} & \kkpic \arrow[u] & \kkpic \arrow[l] \arrow[u] & \kkpic \arrow[l] \arrow[u]   &                         &                                  
    \end{tikzcd}
  \end{center}

  The $3$-multicomplex $\Zzw^3_2(p,q)$ can be depicted as:
  \begin{center}  
    \begin{tikzcd}[sep=scriptsize]
      \cdots\hspace{-2em} & \kkpic            & \kkpic \arrow[l]            & \kkpic \arrow[l] \arrow[lld]            & \kkpic \arrow[l] \arrow[lld]           & \kkpqpic \arrow[lld] \arrow[l] \\
      \cdots\hspace{-2em} & \kkpic\kkpic \arrow[u] & \kkpic\kkpic \arrow[l] \arrow[u] & \kkpic\kkpic \arrow[l] \arrow[u] \arrow[lld] & \kkpic \arrow[lld] \arrow[l] \arrow[u] &                                   \\
      \cdots\hspace{-2em} & \kkpic \arrow[u]  & \kkpic \arrow[l] \arrow[u]  &                                    &                                   &                                  
    \end{tikzcd}
  \end{center}
\end{exmp}

\begin{defn}
  Define the $n$-multicomplexes 
\[\Bbw^n_0(p,q-1) = 0,\qquad  \Bbw^n_1(p,q-1) = \mathbb{D}^n(p,q-1)
\]
 and for $r \geq 2$, define the $n$-multicomplex $\Bbw^n_r(p,q-1)$ to be
  \[
    \Bbw^n_r(p,q-1)=\Zzw^n_{r-1}(p+r-1,q+r-2) \oplus \mathbb{D}^n(p,q-1) \oplus \Zzw^n_{r-1}(p-1,q-1).
  \]
\end{defn}

\begin{lemma}\label{reval}Let $r\geq 0$ and let $p,q \in \Z$ and $2\leq n\leq \infty$.
  \begin{enumerate}
    \item\label{reval1} \label{ls} Giving a morphism of $n$-multicomplexes $\Zzw^n_r(p,q)\to A$  is equivalent to giving an element in $ZW_{r}\bgd{p}{q}(A)$.
    \item \label{ldr} Giving a morphism of $n$-multicomplexes $\Bbw^n_r(p,q)\to A$  is equivalent to giving an element in $BW_{r}\bgd{p}{q}(A)$.
  \end{enumerate}
  Furthermore, these statements are functorial, so that $\Zzw^n_r(p,q)$, $\Bbw^n_r(p,q)$ are representing $n$-multicomplexes for the functors $ZW^{p,q}_r,\, BW^{p,q}_r \colon \ncpx \to \kk\Mod$ respectively.
\end{lemma}

\begin{proof}
  The case $r=0$ in part (\ref{ls}) is immediate from the definition of $\Zzw^n_0(p,q) = \mathbb{D}^n(p,q)$.
  For $r \geq 1$ we proceed inductively: assume $ZW_{r-1}^{p,q} = \ncpx(\Zzw^n_{r-1}(p,q),-)$ as functors $\ncpx \to \kk\Mod$.
  It is an easy check that the $n$-multicomplex morphisms $d_0^*$ and $D_{r-1}^*$ correspond to the natural transformations $d_0$ and $D_{r-1}$ in \cref{P:pullbackfunctors} under the Yoneda embedding $\Yy \colon \ncpx \to \Fun(\ncpx,\kk\Mod)^{\op}$. 
  Furthermore, $\Yy$ takes pushout squares in $\ncpx$ to pullback squares in $\Fun(\ncpx,\kk\Mod)$, hence $ZW_r^{p,q} = \ncpx(\Zzw^n_r(p,q),-)$ by \cref{P:pullbackfunctors}.
  Part (\ref{ldr}) is now immediate from the definition of $BW_{r}\bgd{p}{q}(A)$.
\end{proof}

Lastly, for $r \ge 0$, define $\iota_r \colon \Zzw^n_r(p,q) \to \Bbw^n_r(p,q-1)$ to be the $n$-multicomplex morphism corresponding to the natural transformation $w_{r} \colon BW_{r}\bgd{p}{q-1} \rightarrow ZW_{r}\bgd{p}{q}$ under the Yoneda embedding $\Yy$.
Under these correspondences, a commutative diagram of $n$-multicomplexes of the form
\[\label{ldiag}\xymatrix{
    \Zzw^n_r(p,q)\ar[d]_{\iota_r}\ar[r]&A\ar[d]^f\\
    \Bbw^n_r(p,q-1)\ar[r]_{}&B
}\]
corresponds to a pair $(a,b)$, $a\in ZW_{r}^{p,q}(A)$, $b\in BW_{r}^{p,q-1}(B)$ such that $ZW_r(f)(a)=w_{r}(b)$.

\medskip
The following two results will be useful for constructing our model category structures. 

\begin{lemma}\label{lemretr} Let $2\leq n\leq\infty$. 
  For $r\geq 1$ the $n$-multicomplex $\Zzw^n_0(p,q-1)$ is a retract of $\Bbw^n_r(p,q-1)$ and for $r\geq 2$ the $n$-multicomplex $\Zzw^{n}_{r-1}(p-1,q-1)$ is a retract of $\Bbw^n_r(p,q-1)$.
\end{lemma}

\begin{proof}
  Immediate from the definition of $\Bbw_r^{p,q}$.
\end{proof}

\begin{lemma}\label{lempush}
  Let $2\leq n\leq\infty$. For $r\geq 1$, the diagram
  \[
    \xymatrix{
      \Zzw^n_r(p,q)\ar[d]_{\iota_r}\ar[r]&0\ar[d]\\
  \Bbw^n_r(p,q-1)\ar[r]&\Zzw^n_r(p+r-1,q+r-2) } \]
  is a pushout diagram in $n$-multicomplexes. 
\end{lemma}

\begin{proof}
  By \cref{L:kerwr}, the following diagram is a pullback square in the functor category $\Fun(\ncpx,\kk\Mod)$ for $r \ge 1$.
  \begin{center}
    \begin{tikzcd}
      ZW_r^{p+r-1,q+r-2} \arrow[r] \arrow[d] & BW_r^{p,q-1} \arrow[d,"w_r"] \\
      0 \arrow[r] & ZW_r^{p,q}
    \end{tikzcd}
  \end{center}
  The result now follows by Yoneda's lemma.
\end{proof}

\subsection{Model category structures}\label{S:model}

In this section, we present the model structures on $n$-multicomplexes. We are able to exploit the $r$-cone defined for the case of bicomplexes.

We denote by $C_r$ the bicomplex $\Zzw_r^2(0,0)$. We recall from~\cite{celw} that for $r=0$ it is depicted as a square, and for $r\geq 1$, it is depicted as a staircase graph with $r$ horizontal steps as follows, where each vertex marked $\kkpic$ represents $\kk$, each arrow represents the identity map and the vertex marked $\kkpqpic$ represents $\kk$ in bidegree $(0,0)$.
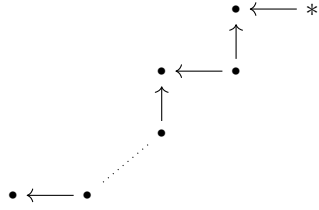
\[
  \begin{tikzcd}[cramped,sep=scriptsize]
                 &                  &                                        & \kkpic                     & \kkpqpic \arrow[l] \\
                 &                  & \kkpic                                 & \kkpic \arrow[u] \arrow[l] &                   \\
                 &                  &  \kkpic \arrow[u] \arrow[ld, no head, dotted] &                            &                   \\
\kkpic           & \kkpic \arrow[l] &                                        &                            &                   \\
\end{tikzcd}
\]

We may write $C_r=\bigoplus_{k=0}^{r-1} \kk \beta_{-k,-k}\oplus\bigoplus_{k=0}^{r-1} \kk\beta_{-k-1,-k}$ with the differentials $d_0,d_1$ indicated by the graph, and $\beta_{i,j}$ a generator of bidegree $(i,j)$.
Since $C_r$  is a bicomplex, we may also view it as an $n$-multicomplex, for $2\leq n\leq\infty$.
Then for any $n$-multicomplex $A$, with $2\leq n\leq\infty$, we have that $C_r\otimes A$ is an $n$-multicomplex, using the symmetric monoidal structure on $\ncpx$.

\begin{prop}\label{Crtrivial}  
  Let $2\leq n\leq \infty$. Let $A$ be an $n$-multicomplex and $r\geq 0$. Then $E_{r+1}(C_r\otimes A)=0$.
\end{prop}

\begin{proof}  
  Proposition 4.29 of \cite{celw} proves that $C_r$ is $r$-contractible in the sense that the identity map of $C_r$ is $r$-homotopic (in the category of bicomplexes but also in that of  multicomplexes) to 0. 
  As a corollary, for any multicomplex (and thus for any $n$-multicomplex $A$), $C_r\otimes A$ is $r$-contractible, hence by~\cite[Proposition 3.24]{celw18},
  $E_{r+1}(C_r\otimes A)=0$.
\end{proof}

\begin{prop}\label{coneetfib} 
  Let $p,q\in\Z$ and $2\leq n\leq \infty$. Let $A$ be an $n$-multicomplex  and $r\geq 0$. The projection morphism $\phi_r:C_r\otimes A\rightarrow A$  has the property that $ZW^{p,q}_k(\phi_r)$ is surjective for $0\leq k\leq r$. 
\end{prop}

\begin{proof} The case $r=0$ is trivial. Let us assume $r\geq 1$.  
  Let $(a_0,a_1,\ldots,a_{r-1})$ be an element of $ZW_r^{p,q}(A)$, with $a_i\in A^{p-i,q-i}$.
  We have

  \[\sum_{i+j=l} (-1)^id_ia_j = 0~\text{for}~0\leq l\leq r-1.\]

  For $0\leq k\leq r-1$, we define the element

  \[X_k=\sum_{i=0}^k \beta_{-i, -i}\otimes a_{k-i}  \in (C_r\otimes A)^{p-k,q-k}.\]

  Let us prove that  $(X_0,\ldots,X_{r-1})$ is an element of $ZW_r(C_r\otimes A).$ Fix $0\leq l\leq r-1$ and compute
  \begin{align*}
    \sum_{i=0}^l (-1)^id_iX_{l-i}
    = & \sum_{i=0}^l (-1)^id_i \big(\sum_{j=0}^{l-i} \beta_{-j,-j}\otimes a_{l-i-j}\big) \\
    =& \sum_{j=0}^{l} (d_0\beta_{-j,-j})\otimes a_{l-j}-\sum_{j=0}^{l-1} (d_1\beta_{-j,-j})\otimes a_{l-1-j}\\
    &\quad+\sum_{j=0}^l {(-1)^{j}}\beta_{-j,-j}\otimes \big(\sum_{i=0}^{l-j} (-1)^{i} d_ia_{l-i-j}\big)\\
    =& \sum_{j=1}^{l} \beta_{-j,-j+1}\otimes a_{l-j}-\sum_{j=0}^{l-1} \beta_{-j-1,-j}\otimes a_{l-1-j}=0.
  \end{align*}
  Hence, the induced map $ZW_r(\phi_r)$ on $ZW_r(C_r\otimes A)$ satisfies
  \[
    ZW_r(\phi_r)(X_0,\ldots,X_{r-1})=(a_0, \dots, a_{r-1}).
  \]
  Note that since $(X_0, \dots, X_k)\in ZW_k(C_r\otimes A)$ is defined from the data $(a_0,\ldots,a_k)$, the same proof applies to $ZW_k(\phi_r)$,
  for $0\leq k\leq r$. 
\end{proof}

\begin{rem}\label{R:crinfty} Let $C_r^\infty$ be the multicomplex $\kk e_{0,0}\oplus\kk e_{-r,1-r}$ with only non trivial differential $d_r(e_{0,0})=e_{-r,1-r}$.
  We have that $C_r^\infty$ is an $r$-contractible multicomplex, with $h_0(e_{-r,1-r})=e_{0,0}$ satisfying $d_rh_0+h_0d_r=1_{C_r^\infty}$. Hence, for any multicomplex $Y$, $C_r^\infty\otimes Y$ is $r$-contractible.
  In addition the projection $\pi \colon C_r^\infty\otimes Y\rightarrow Y$ induced by the projection of $C_r^\infty$ onto $\kk e_{0,0}$ satisfies $ZW_s(\pi)$ is surjective for all $0\le s\leq r$: 
  it is easy to see that if $(a_0,\ldots,a_{s-1})\in ZW_s(Y)$ then $(e_{0,0}\otimes a_0,\ldots, e_{0,0}\otimes a_{s-1})\in ZW_s(C_r^\infty\otimes Y)$.
\end{rem}

\begin{defn}\label{defIrJr} Let $2\leq n\leq\infty$.
  For $r\geq 0$, consider the sets of morphisms of $n$-multicomplexes
  \[I^n_r=\left\{\xymatrix{\Zzw^n_{r+1}(p,q)\ar[r]^-{\iota_{r+1}} & \Bbw^n_{r+1}(p,q-1)}\right\}_{\substack{p,q\in\Z}} 
    \text{ and } 
  J^n_r=\left\{\xymatrix{0\ar[r]& \Zzw^n_{r}(p,q)}\right\}_{\substack{p,q\in\Z}}.\]
\end{defn}

\begin{prop}\label{r-Hovey-234} For each $r\geq 0$, a map $f$ is $J^n_r$-injective if and only if $ZW_r(f)$ is surjective.
\end{prop}

\begin{proof} This follows from $(\ref{reval1})$ of Lemma \ref{reval}.
\end{proof} 

\begin{prop}\label{trivfib_r}
  For all $r\geq 0$ and $2\leq n\leq \infty$, we have $\cl{I^n_r}{inj}= \Ee^n_r\cap \cl{J^n_0}{inj}\cap \cl{J^n_r}{inj}$.
\end{prop}

\begin{proof}  
The proof proceeds exactly like that of~\cite[Proposition 4.35]{celw}, the corresponding result in the bicomplex case $n=2$, using Lemmas~\ref{L:kerwr}, \ref{reval}, \ref{lemretr},
\ref{lempush} and Remark~\ref{mainobs}.
\end{proof}

\begin{prop}\label{Jr-cof} 
  For all $r\geq 0$, $2\leq n\leq \infty$ and all $0\leq k\leq r$ we have $\cl{J^n_k}{cof}\subseteq \Ee^n_r$.
\end{prop}

\begin{proof} Let $r\geq 0$ and $0\leq k\leq r$ and $f \colon X\rightarrow Y\in \cl{J^n_k}{cof}$.
  Consider the following diagram.
  \[
    \xymatrix{ X\ar[r]^-{\vmat{1_X}{0}}\ar[d]_f & X\oplus (C_r\otimes Y)\ar[d]^{\hmat{f}{\phi_r}}\\ Y\ar[r]^{=} & Y}
  \]

  From Propositions~\ref{coneetfib} and~\ref{r-Hovey-234} the right-hand vertical map is $J^n_k$-injective so there is a lift in the diagram. From  Proposition~\ref{Crtrivial} one has $E_{r+1}(C_r\otimes Y)=0$.
  Applying the functor $E_{r+1}$ to the diagram, we see that $E_{r+1}(f)$ is an isomorphism.  
  Note that in case $n=\infty$ the proof also holds using  $C_r^\infty$ (instead of $C_r$) and Remark \ref{R:crinfty}.
\end{proof}

\begin{thm}\label{modelZ}
  For every $r\geq 0$ and $2\leq n\leq \infty$, 
  the category $\ncpx$ admits a right proper cofibrantly generated model structure, where: 
  \begin{enumerate}
    \item weak equivalences are $E_r$-quasi-isomorphisms,
    \item fibrations are morphisms of $n$-multicomplexes $f \colon A\to B$ 
      such that $f$ and  $ZW_r(f)$ are bidegree-wise surjective, and
    \item $I^n_r$ and $J^n_0\cup J^n_r$ are the sets of generating cofibrations and generating trivial cofibrations respectively.
  \end{enumerate}
\end{thm}

\begin{proof} 
  The proof is standard (see, for example, the proof of~\cite[Theorem 3.14]{celw}) and uses Propositions~\ref{r-Hovey-234}, Proposition~\ref{Jr-cof} and Proposition~\ref{trivfib_r}.  
\end{proof}

As in the bicomplex case, in certain situations it may be easier to characterize fibrations if they are described in terms of surjectivity of $E_i$ instead of $ZW_r$.

\begin{defn}
  Let $(I^n_r)'$ and $(J^n_r)'$ be the sets of morphisms of $\ncpx$ given by 
  \[(I^n_r)':=\cup_{k=1}^{r-1} J^n_k\cup I^n_r\text{ and }
  (J^n_r)':=\cup_{k=0}^r J^n_k.\]
\end{defn}

The proof of the following result is analogous to that for bicomplexes~\cite[Theorem 4.39]{celw}.

\begin{thm}\label{modelE} 
  For every $r\geq 0$ and $2\leq n\leq\infty$, 
  the category $\ncpx$ admits a right proper cofibrantly generated model structure, denoted $(\ncpx)_r$, where: 
  \begin{enumerate}
    \item weak equivalences are $E_r$-quasi-isomorphisms,
    \item fibrations are morphisms of $n$-multicomplexes $f \colon A\to B$ 
      such that $E_i(f)$ is  bidegree-wise surjective for every $0\leq i\leq r$, and
    \item $(I^n_r)'$ and $(J^n_r)'$ are the sets of generating cofibrations and generating trivial cofibrations respectively. \qed
  \end{enumerate}
\end{thm}

\begin{defn}
We refer to the model structure $(\ncpx)_r$ of Theorem~\ref{modelE} as the \emph{$r$-model structure}. The terms \emph{$r$-fibrant}, \emph{$r$-cofibrant} and \emph{$r$-trivial} all refer to the corresponding notions in this model structure.
\end{defn}

\begin{rem} 

  Note that the generating (trivial) cofibrations of the  model structure of Theorem~\ref{modelZ} form a subclass of the 
generating (trivial) cofibrations of the model structure of Theorem~\ref{modelE}.
Moreover, these two model structures have the same weak equivalences.
  Thus, for each $n$ with $2\leq n\leq\infty$ and each $r\geq 0$, the identity functors
  give a Quillen equivalence between  $(\ncpx)_r$ and  $\ncpx$ with the model structure of Theorem~\ref{modelZ}.
\end{rem}

\section{Relationships between model category structures}
\label{S:relationships}

In order to compare our model structures on $n$-multicomplexes as $n$ varies, in this section we reinterpret
$n$-multicomplexes as modules over a graded associative algebra in the category of vertical bicomplexes.

\subsection{Monoids in vertical bicomplexes}

Recall from Section~\ref{S:preliminaries} that the category $\cpx{1}=\vbic$ has as objects vertical bicomplexes, and that it is a symmetric monoidal category. A monoid $(M,\delta_0)$ in this category is a vertical bicomplex endowed with a unital and associative multiplication $M\otimes M\rightarrow M$ compatible with the differential $\delta_0$. In other words, it is a unital bigraded (associative, not necessarily commutative) $\kk$-algebra, endowed with a derivation of algebras $\delta_0$ of bidegree 
(0,1) such that $\delta_0^2=0$. For simplicity, we call such an object a {\sl dg algebra}. This is only a slight abuse of terminology -- this differs from the usual notion just by having an extra grading.

Consider $\kk\langle d_1,d_2,\ldots,d_i,\ldots\rangle$, the free bigraded associative $\kk$-algebra generated by the bigraded set $ \{d_i,i\geq 1\}$, with $d_i$ of bidegree $(-i,1-i)$.

For $k\geq 1$, we consider the following element of $\kk\langle d_1,d_2,\ldots,d_i,\ldots\rangle$:
\[S_k=\sum_{\substack{i+j=k\\   i,j\geq 1}}(-1)^{i+1}d_id_j,\]
in bidegree $(-k, 2-k)$, 
with the convention that $S_1=0$.

Since  $\kk\langle d_1,d_2,\ldots,d_i,\ldots\rangle$ is a free associative algebra, a derivation $\delta_0$ on  $\kk\langle d_1,d_2,\ldots,d_i,\ldots\rangle$ is determined by its values on the generators $d_i$.
Set, for $i\geq 1$,
\[\delta_0(d_i)=S_i.\]
Proving that $\delta_0^2=0$  amounts to proving that $\delta_0(S_i)=0$, which is standard.

For $n\geq 1$, let  $I_n$ be the two sided ideal of  $\kk\langle d_1,d_2,\ldots,d_i,\ldots\rangle$ generated by the elements $S_k$ and $d_k$ for $k\geq n$. The definition of $\delta_0$ shows that this ideal is compatible with the differential.
By convention $I_\infty=\{0\}$.

\begin{defn}
  Let $\mathcal C_\infty$ be the dg algebra $\kk\langle d_1,d_2,\ldots,d_i,\ldots\rangle$ endowed with differential $\delta_0$ as above. And for $n\geq 1$, let $\mathcal C_n$ be the dg algebra
  \[\mathcal C_n=(\mathcal C_\infty/I_n,\delta_0).\]
\end{defn}

For $1\leq n\leq l \leq\infty$, we have $I_l\subset I_n$ and thus a surjective morphism of dg algebras  
\[\Phi_{l,n}\colon \mathcal{C}_{l} \to \mathcal{C}_n.\]

\begin{prop}\label{P:quasiiso-monoid} If $2\leq n\leq l \leq\infty$, then $\Phi_{l,n}\colon \mathcal{C}_{l} \to \mathcal{C}_n$ is a quasi-isomorphism of vertical bicomplexes.
\end{prop}

\begin{proof} For $2\leq m\leq n\leq l\leq\infty$ we have $\Phi_{l,m}=\Phi_{n,m}\circ \Phi_{l,n}$. 
  By the two-out-of-three property of quasi-isomorphisms, it is enough to prove that the maps $\Phi_{\infty,n} \colon \mathcal C_\infty\rightarrow \mathcal C_n$ are quasi-isomorphisms for all $2\leq n< \infty$. 

  For $n=2$, we have $\mathcal C_2=\kk\langle d_1\rangle/(d_1^2)$ with $\delta_0(d_1)=0$. 
  Hence it is enough to prove that the induced map on the homology with respect to $\delta_0$,
  $H^{*,*}(\Phi_{\infty,2}) \colon H^{*,*}(\mathcal C_\infty)\rightarrow H^{*,*}(\mathcal C_2)$ is an isomorphism. 
  In order to do so we build a homotopy $h \colon \mathcal C_\infty\rightarrow \mathcal C_\infty$. 
  Any element in $\mathcal C_\infty$ is a linear combination with coefficients in $\kk$ of words of the form $d_{i_1}\ldots d_{i_k}$ with $i_j\geq 1$. 
  The empty word corresponds to $1_\kk$ and we define $h(1_\kk)=0$.
  Let $h$ be the $\kk$-linear map determined by
  \[
    h(d_{i_1}\ldots d_{i_k})=\begin{cases} 0, & \text{ if } k=1 \text { or } i_1>1,\\
    d_{i_2+1}\ldots d_{i_k}, & \text{ if } k>1 \text { and } i_1=1.\end{cases}
  \]
  Note that for any word $w$ we have $h(S_iw)=d_iw$ for $i\geq 2$.
  Let us compute:
  \begin{align*}
    &(\delta_0h+h\delta_0)(1_\kk)=0, \quad (\delta_0h+h\delta_0)(d_1)=0, \\
    &(\delta_0h+h\delta_0)(d_i)=h(S_i)=d_i, \qquad\text{for $i\geq 2$.}
  \end{align*}
  For $k\geq 2$,
  \begin{align*}
    (\delta_0h&+h\delta_0)(d_1d_{i_2}\ldots d_{i_k})\\
    &=
    \delta_0(d_{i_2+1}d_{i_3}\ldots d_{i_k})+\sum_{j=2}^{k}(-1)^{i_2+1+\ldots+i_{j-1}+1} h(d_1d_{i_2}\ldots \delta_0(d_{i_j})\ldots d_{i_k})\\
    &=
    \sum_{j=3}^{k}(-1)^{i_2+2+\ldots+i_{j-1}+1} d_{i_2+1}\ldots S_{i_j}\ldots d_{i_k}+\sum_{j=3}^{k}(-1)^{i_2+1+\ldots+i_{j-1}+1} d_{i_2+1}\ldots S_{i_j}\ldots d_{i_k}\\
    &\qquad +S_{i_2+1}d_{i_3}\ldots d_{i_k}+\sum_{\substack{u+v=i_2\\ u,v\geq 1}}(-1)^{u+1}d_{u+1}d_vd_{i_3}\ldots d_{i_k}\\
    &=d_1d_{i_2}d_{i_3}\ldots d_{i_k},
  \end{align*}
  and for $i_1>1$ 
  \[
    (\delta_0h+h\delta_0)(d_{i_1}d_{i_2}\ldots d_{i_k})=
  h(S_{i_1}d_{i_2}\ldots d_{i_k})=d_{i_1}d_{i_2}\ldots d_{i_k}.\]
  So $H^{p,q}( \mathcal C_\infty)=0$ for every $(p,q)\not\in\{(0,0),(-1,0)\}$.
  In addition
  \[(\mathcal C_\infty)^{0,0}=\kk,\ \mathcal C_{\infty}^{-1,0}=\kk d_1,\ \mathcal C_\infty^{0,-1}= \mathcal C_\infty^{-1,-1}=0,\] 
  hence $\Phi_{\infty,2}$ is a quasi-isomorphism.

  Let us prove that for any $n\geq 3$ we have $h(I_n)\subseteq I_n.$
  Any element in $I_n$ is a sum of elements of the form $abc$, with $a,c\in\mathcal C_\infty$ and $b=S_k$ or $b=d_k$ for some $k\geq n$.
  If $a\not=1_\kk$ or $a\not=d_1$ then $h(abc)=h(a)bc\in I_n$. \\
  Assume $a=1_\kk$ and let $k\geq n\geq 3$.
  \begin{itemize}
    \item If $b=d_k$  then $h(d_kc)=0.$
    \item If $b=S_k$  then $h(S_kc)=d_kc\in I_n$.
  \end{itemize}
  Assume $a=d_1$.
  \begin{itemize}
    \item If $b=d_k$  then $h(d_1d_kc)=d_{k+1}c \in I_n,$ since $k+1\geq n+1\geq n$.
    \item If $b=S_k$  then 
      \[
        h(d_1S_kc)=\sum_{\substack{u+v=k\\ u,v\geq 1}}(-1)^{u+1} h(d_1d_ud_vc)=\sum_{\substack{u+v=k\\ u,v\geq 1}}(-1)^{u+1} d_{u+1}d_vc
      =-S_{k+1}c+d_1d_{k+1}c\in I_n.\]
  \end{itemize}
  The quotient map $\Phi_{\infty,n} \colon \mathcal C_\infty\rightarrow \mathcal C_n$ has kernel $I_n$ and $h \colon I_n\rightarrow I_n$ is a homotopy from the identity of $I_n$ to $0$. Hence $I_n$ is contractible and the morphism is a quasi-isomorphism.
\end{proof}

\begin{rem} The morphism $\Phi_{\infty,2} \colon \mathcal C_\infty\rightarrow \mathcal C_2$ corresponds to the Koszul resolution of the operad of dual numbers $\mathcal C_2$ (see for example~\cite[10.3.16]{lv}), hence it is a quasi-isomorphism.
  The proof given here via the homotopy $h$ is not a consequence of this result, however, and this method has been chosen because it allows us to treat the case of general $n$.
\end{rem}

\subsection{Quillen equivalences}\label{S:quillenequivalence}

\begin{prop}\label{P:Cn-modules}  For $1\leq n\leq\infty$, the category of\/ $\mathcal C_n$-modules in vertical bicomplexes is isomorphic to the category of $n$-multicomplexes.
\end{prop}

\begin{proof}
  In the category of vertical bicomplexes a (left) $\mathcal C_\infty$-module is a bigraded $\kk$-module $M$ endowed with a differential $d_0^M$ of bidegree $(0,1)$ 
  together with an action $\lambda \colon \mathcal C_\infty\otimes M\rightarrow M$ compatible with the differentials $\delta_0$ and $d_0^M$. 
  Since $\mathcal C_\infty$ is free as a bigraded $\kk$-algebra the action is determined by its values on $d_i, i\geq 1$. We denote by 
  $d_i^M \colon M\rightarrow M$ the map that associates $\lambda(d_i\otimes m)$ to $m$. The compatibility with the differentials gives that
  \[d_0^Md_n^M=\sum_{i+j=n,i,j\geq 1}(-1)^{i+1}d_i^Md_j ^M+(-1)^{1-n}d_n^Md_0^M,\]
  that is, $M$ is a multicomplex. 
  In addition morphisms of  $\mathcal C_\infty$-modules are morphisms of multicomplexes. 
  This completes the proof for $n=\infty$.
  A (left) $\mathcal C_n$-module is a (left) $\mathcal C_\infty$-module $M$ such that $d_i^M=0$ for all $i\geq n$, hence an $n$-multicomplex.
\end{proof}

As a corollary, the dg algebra morphisms $\Phi_{l,n}\colon \mathcal{C}_{l} \to \mathcal{C}_n$, for $1\leq n\leq l\leq \infty$ 
induce pairs of adjoint functors 
\begin{center}
  \begin{tikzcd}
    \cpx{n}= \mathcal{C}_{n}\Mod \arrow[r, "i_{l,n}"', shift right] & \mathcal{C}_{l}\Mod=\cpx{l} \arrow[l, "p_{l,n}"',shift right] 
  \end{tikzcd}
\end{center}
where the right adjoint $i_{l,n}$ is the restriction of scalars functor and the left adjoint $p_{l,n} (M) = \mathcal C_n \otimes_{\mathcal C_{l}} M$ is the extension of scalars functor. Note that if $M$ is an $n$-multicomplex, then $i_{l,n}(M)$ is the $l$-multicomplex $M$ with $d_n=\ldots=d_{l-1}=0$.

Recall that we write $(\cpx{n})_r$ for the category of $n$-multicomplexes with the $r$-model structure of Theorem~\ref{modelE}.

\begin{thm}\label{Qequiv}  For\/ $2\leq n\leq l\leq \infty$ and $r\geq 0$ the adjunction
  \[
    \begin{tikzcd}
      (\cpx{n})_r \arrow[r, "i_{l,n}"', shift right] &(\cpx{l})_r \arrow[l, "p_{l,n}"',shift right] 
    \end{tikzcd}
  \]
  is a Quillen equivalence.
\end{thm}

\begin{proof}
  It is a Quillen adjunction from Theorem~\ref{modelE}, for the right adjoint preserves fibrations and trivial fibrations.
  Note that the right adjoint reflects weak equivalences and that all objects are fibrant. 
  Hence to establish a Quillen equivalence it is enough to prove that for any $r$-cofibrant object $M$ in $\cpx{l}$, the
  unit of the adjunction $M \to i_{l,n}p_{l,n} M$ is an $E_r$-quasi-isomorphism (see~\cite[Corollary 1.3.16]{Hovey}).

  Recall that  any $r$-cofibrant object is $0$-cofibrant.  
  Thus, if the unit of the adjunction is an $E_0$-quasi-isomorphism for any $0$-cofibrant object, then it is an $E_r$-quasi-isomorphism for any $r$-cofibrant object, and it is enough to treat the case $r=0$.

  Let us prove that the adjunction is a Quillen equivalence for $r=0$.  
  The model category structure $(\cpx{n})_0$ corresponds to the transferred model category structure along the adjunction
  \[
    \begin{tikzcd}
      \cpx{n}= \mathcal{C}_{n}\Mod \arrow[r, "U_n"', shift right] & \vbic \arrow[l, "\mathcal C_n\otimes-"',shift right],
    \end{tikzcd}
  \]
  where the right adjoint $U_n$ is the forgetful functor and the model category structure on $\vbic$ coincides with the projective model structure on $\Z$-graded cochain complexes, that is, weak equivalences are quasi-isomorphisms with respect to the bidegree $(0,1)$ differential  $d_0$, fibrations are bidegreewise surjective morphisms. 
  A standard result (see~\cite[Proposition 11.2.10]{Fresse09}) states that a morphism of dg algebras $\alpha \colon R \to S$ induces a Quillen adjunction between the categories of $R$-modules and $S$-modules (with the transferred model structure from $\vbic$ as seen above) through the restriction and extension of scalars functors, and this is a Quillen equivalence if (and only if) $\alpha$ is a quasi-isomorphism.
  Hence, Proposition \ref{P:quasiiso-monoid} implies that the Quillen adjunction
  \[
    \begin{tikzcd}
      (\cpx{n})_0 \arrow[r, "i_{l,n}"', shift right] &(\cpx{l})_0 \arrow[l, "p_{l,n}"',shift right] 
    \end{tikzcd}
  \]
  is a Quillen equivalence, that is, $M\rightarrow  i_{l,n}p_{l,n} M$  is an $E_0$-quasi-isomorphism for every $0$-cofibrant object $M$ in $\cpx{l}$.
\end{proof}

\begin{rem} In the previous proof the model category structure considered in $\vbic$ is precisely $(\cpx{1})_0$. The adjunction
  \[
    \begin{tikzcd}
      (\cpx{1})_0 \arrow[r, "i_{2,1}"', shift right] &(\cpx{2})_0 \arrow[l, "p_{2,1}"',shift right] 
    \end{tikzcd}
  \]
  is a Quillen adjunction, however it is not a Quillen equivalence. Indeed, $\Zzw^2_{1}(0,0)$ is $0$-cofibrant in $\cpx{2}$ and the unit of the adjunction for this object is the projection onto the $(0,0)$-coordinate
  \[\Zzw^2_{1}(0,0) \to \kk^{0,0}\]
  which is not an $E_0$-quasi-isomorphism.
\end{rem}

\section{Model structures on bounded multicomplexes}
\label{S:bounded}

In this section, we will apply the transfer theorem to give model structures on certain categories of bounded $n$-multicomplexes. 
We obtain such transferred model structures on $(-\N,\Z)$-graded $n$-multicomplexes for 
all $r\geq 0$ and on $(\Z,\N)$-graded multicomplexes for $r=0$.
Our exposition of the transfer principle follows~\cite[Sections 2.5--2.6]{bm03}.

\begin{thm*}
  Let $\Mm$ be a model category cofibrantly generated by the sets\/ $I$ and\/ $J$ of generating cofibrations and generating trivial cofibrations respectively.
  Let $\Cc$ be a category with finite limits and small colimits.
  Let
  \begin{center}
    \begin{tikzcd}
      \Mm \arrow[r,shift left=.5ex,"L"] &
      \Cc \arrow[l,shift left=.5ex,"R"]
    \end{tikzcd}
  \end{center}
  be a pair of adjoint functors. 
  Define a map $f$ in\/ $\Cc$ to be a weak equivalence \textup(respectively fibration\textup) if $R(f)$ is a weak equivalence \textup(respectively fibration\textup).
  These two classes determine a model category structure on\/ $\Cc$ cofibrantly generated by $L(I)$ and $L(J)$ provided that:
  \begin{enumerate}
    \item\label{transfer-small} The sets $L(I)$ and $L(J)$ permit the small object argument.
    \item\label{transfer-functorial} $\Cc$ has a functorial fibrant replacement and a functorial path object for fibrant objects.
  \end{enumerate}
  Furthermore, with this model structure on\/ $\Cc$, the adjunction $L \dashv R$ becomes a Quillen adjunction.
\end{thm*}

Recall that a path object for $X$ is a factorisation of the diagonal map $X \longrightarrow X \times X$ into a weak equivalence followed by a fibration $X\overset{\sim}\longrightarrow P(X) \twoheadrightarrow X\times X$. 
To apply the transfer theorem, we first need to show the existence of $r$-path objects for $n$-multicomplexes. 
For this, we adapt~\cite[Section 5]{celw18} to our context.

\subsection{Path objects for $n$-multicomplexes}

As with the $r$-cone, we start with constructions for bicomplexes and then extend to $n$-multicomplexes using the tensor product.

\begin{defn}[\cite{celw}]\label{cyl0}
  For $r=0$, we define the $0$-\textit{path} $\Lambda_0$ as the bicomplex
  \[\xymatrix{
      \kk^{0,1}\\
      ( \kk\oplus \kk)^{0,0}  \ar[u]_{{\hmat{-1}{1}}}.\\
  }\] 
  For $r\geq 1$, define the $r$-\textit{path} $\Lambda_r$ as the bicomplex whose underlying bigraded module is 
  $\kk\bgd{0}{0}\oplus \Zzw^2_{r}(0,0)$ and whose differentials coincide 
  with those of $\Zzw^2_{r}(0,0)$ except for $d_1\bgd{0}{0}$ which is:
  \[\xymatrix{
      \Zzw^2_{r}(0,0)\bgd{-1}{0}=\kk\bgd{-1}{0}&&
      (\kk\oplus \Zzw^2_{r}(0,0))\bgd{0}{0}=(\kk\oplus\kk)\bgd{0}{0}
    \ar[ll]_-{\hmat{-1}{1}}}.
  \] 
\end{defn}

\begin{exmp}
  The $1$-path $\Lambda_1$ is the bicomplex given by
  \[
    \kk\bgd{-1}{0}\stackrel{\hmat{-1}{1}}{\longleftarrow}(\kk\oplus\kk)\bgd{0}{0}.
  \]
  The $2$-path $\Lambda_2$ is given by
  \begin{center}
    \begin{tikzpicture}
      scale=0.15
      \matrix (m) [matrix of math nodes, row sep=2em,
      column sep=2em]{
        &  \kk^{-1,0} &(\kk\oplus\kk)^{0,0} \\
      \kk^{-2,-1} &\kk^{-1,-1}& \\};
      \path[-stealth]
      (m-1-3) edge  node[midway,above] {$\hmat{-1}{1}$} (m-1-2) 
      (m-2-2) edge   node[midway,right] {$\idmat$} (m-1-2) edge   node[midway,above] {$\idmat$} (m-2-1);
    \end{tikzpicture}
  \end{center}
\end{exmp}

More generally, we write
\[\Lambda_r=\kk \beta_-\oplus \bigoplus_{i=0}^{r-1} \kk \beta_{-i,-i} \bigoplus_{i=0}^{r-1} \kk \beta_{-i-1,-i}\]
where $\beta_{u,v}$ has bidegree $(u,v)$ and $\beta_-$ has bidegree $(0,0)$, with nonzero differentials given by
\[d_1(\beta_{0,0})=-d_1(\beta_-)=\beta_{-1,0},\; d_0(\beta_{-i,-i})=\beta_{-i,1-i}, \ d_1(\beta_{-i,-i})=\beta_{-i-1,-i},\]
for $1\leq i\leq r-1$.

\begin{lemma} For $r\geq 1$, there is an isomorphism of bicomplexes 
  $\varphi_r \colon \Lambda_r\rightarrow\kk\bgd{0}{0}\oplus C_r$ where
  $C_r=\Zzw^2_{r}(0,0)$ has been defined in Section \ref{S:model}.
\end{lemma}

\begin{proof} Let us keep the notation $\beta_{u,v}$ for both the generators of $\Lambda_r$ and $C_r$ and let $e$ be a generator of $\kk\bgd{0}{0}$. The map of bigraded modules $\varphi_r\colon \Lambda_r\rightarrow \kk\bgd{0}{0}\oplus C_r$ which associates $e-\beta_{0,0}$ to $\beta_-$ and
  $\beta_{u,v}$ to $\beta_{u,v}$ for $(u,v)\in\{(-i,-i), (-i-1,-i), 0\leq i\leq r-1\}$ is an isomorphism of bicomplexes since $\varphi_r d_0(\beta_-)=0=d_0(e-\beta_{0,0})$ and $\varphi_rd_1(\beta_-)=-\beta_{-1,0}=d_1(e-\beta_{0,0})$.
\end{proof}

Let us consider the following morphisms of bicomplexes
\begin{center}
  \begin{tikzcd}
    \kk\bgd{0}{0} \arrow[r, "\iota"] & \Lambda_r \arrow[rr, "\pi=\partial_-+\partial_+"] && (\kk\oplus \kk)\bgd{0}{0}
  \end{tikzcd}
\end{center}
where $\iota$ sends $e$ to $\beta_-+\beta_{0,0}$ and $\partial_-$ is the projection onto $\kk \beta_-$ and $\partial_+$ is the projection onto $\kk \beta_{0,0}$.

\begin{prop} \label{rpath}
  For $r\geq 0$, 
  \[\iota\colon \kk\bgd{0}{0}\rightarrow \Lambda_r\]
  is an $r$-homotopy equivalence.
\end{prop}

\begin{proof}  If $r\geq 1$, since an isomorphism is an $r$-homotopy equivalence, it is enough to prove that the composite $\varphi_r\iota=1_\kk\oplus 0 \colon \kk\bgd{0}{0}\oplus 0\rightarrow \kk\bgd{0}{0}\oplus C_r$  is a $r$-homotopy equivalence, which is a direct consequence of the $r$-contractibility of $C_r$ proven in Proposition 4.29 of \cite{celw}. Similarly, if $r=0$, then the bicomplex
  \[\xymatrix{
      \kk^{0,1}\\
      \kk^{0,0}  \ar[u]_{1}\\
  }\] 
  is $0$-contractible and the proof follows.
\end{proof}

\begin{defn} For $A$ an $n$-multicomplex, the \emph{$r$-path object} $P_r(A)$ is the $n$-multicomplex
  $\Lambda_r\otimes A$.
  We denote by $\iota_A$ and $\pi_A$ the maps $\iota\otimes 1_A$ and $\pi\otimes 1_A$ so that the diagonal of $A$ factors as
  \[\xymatrix{  A \ar[r]^-{\iota_A}&P_r(A) \ar[r]^-{\pi_A}&A\oplus A}.\]
This construction is functorial, with $P_r(f)=1_{\Lambda_r}\otimes f \colon P_r(A)\to P_r(B)$, for $f \colon A\to B$ a morphism of 
$n$-multicomplexes.
\end{defn}

\begin{rem}\label{R:P_r} As a bigraded module we have
  \begin{align*}
    P_0(A)^{p,q}=&A^{p,q}\oplus A^{p,q-1}\oplus A^{p,q} \\
    P_r(A)^{p,q}=&A^{p,q}\oplus\bigoplus_{i=0}^{r-1} A^{p+i,q+i}\oplus
    {\bigoplus_{i=0}^{r-1} }A^{p+i+1,q+i},\ \text{ for } r\geq 1.
  \end{align*}
\end{rem}

\begin{prop}\label{pathspace} 
  Let $A$ be an $n$-multicomplex and $r\geq 0$.
  The path object $P_r(A)$ is an $r$-path object for $A$. 
  Indeed, the map $\iota_A \colon A\longrightarrow P_r(A)$ is an $r$-homotopy equivalence, hence an $E_r$-quasi-isomorphism and the map $\pi_A \colon P_r(A)\rightarrow A\oplus A$ is an $r$-fibration in the model structure of Theorem~\ref{modelE}.
\end{prop}

\begin{proof} That $\iota_A$ is an $r$-homotopy equivalence is a direct consequence of Proposition \ref{rpath}. For the second assertion, the case $r=0$ is trivial and for $r\geq 1$, we consider the following commutative diagram of $n$-multicomplexes
 \begin{center}
     \begin{tikzcd}
	P_r(A)\arrow[d,"\varphi_r\otimes 1_A"] \arrow[r, "\pi_A"] 
& A\oplus A\arrow[d,"\mbox{\tiny{$\begin{pmatrix} 1_A \amsamp 0\\-1_A\amsamp 1_A\end{pmatrix}$}}"] \\
  A\oplus (C_r\otimes A)\arrow[r, "1_A\oplus\phi_r"] & A\oplus A
\end{tikzcd} 
\end{center}

  The vertical maps are isomorphisms, hence $\pi_A$ is an $r$-fibration if and only if $1_A\oplus\phi_r$ is an $r$-fibration, which is so by Proposition~\ref{coneetfib} together with Remark~\ref{mainobs}.
\end{proof}

\begin{rem}
A path object for $n$-multicomplexes when $n=\infty$ is given in \cite[Section 3.4]{celw18}.
\end{rem}

\subsection{Model structures on bounded $n$-multicomplexes}

For $2\leq n\leq \infty$, recall that $\ncpx$ denotes the category of $(\Z,\Z)$-graded $n$-multicomplexes  of $\kk$-modules. 
The categories of $(-\N,\Z)$-graded (left half-plane) and  $(\Z,\N)$-graded (upper half-plane) $n$-multicomplexes of $\kk$-modules will be denoted by $\nCh_{\text{-}\N,\Z}$ and $\nCh_{\Z,\N}$,
 respectively.

By  Proposition~\ref{P:Cn-modules}, the category of $n$-multicomplexes is isomorphic to the category of $\mathcal{C}_n$-modules in vertical bicomplexes, previously denoted $\mathcal{C}_n\Mod$. In this section, we will write
 $(\mathcal{C}_n\Mod)_{\Z,\Z}$ when we want to emphasize the $(\Z,\Z)$-grading.

Similarly, the category $ \nCh_{\text{-}\N,\Z}$ is isomorphic to the category of $\mathcal{C}_n$-modules $M$ in vertical bicomplexes concentrated in bidegrees lying in the left half-plane (i.e., with $M^{p,q}=0$ if $p>0$), where the latter is denoted by  $(\mathcal{C}_n\Mod)_{\text{-}\N,\Z}$.

We show that the  inclusion functor from $\nCh_{\text{-}\N,\Z}$ to $\ncpx$ has a left adjoint by showing that the corresponding  inclusion functor from 
$ (\mathcal{C}_n\Mod)_{\text{-}\N,\Z}$ to  $(\mathcal{C}_n\Mod)_{\Z,\Z}$ has a left adjoint.

Let $2\leq n\leq \infty$ and let  $(M, d_0^M, \lambda)$ be a $\mathcal{C}_n$-module, where  $\lambda\colon \mathcal{C}_{n}\otimes M\longrightarrow M$ denotes the module action.
Let $M_{\leq 0}$ and $M_{>0}$ denote the vertical bicomplexes given by
\[
M_{\leq 0}^{p,q}=\begin{cases}
0&\text{ if } p>0\\
M^{p,q} & \text{ if } p\leq 0,
\end{cases}\qquad\text{ and }\qquad
M_{> 0}^{p,q}=\begin{cases}
M^{p,q}&\text{ if } p> 0\\
0& \text{ if } p\leq 0.
\end{cases}
\]
It is clear that $M_{\leq 0}$ is a $\mathcal C_n$-submodule of $M$, that $M_{>0}$ is not, but $\lambda(\mathcal C_n\otimes M_{>0})$ is. Hence the intersection $\lambda(\mathcal{C}_{n}\otimes M_{>0})\cap M_{\leq 0}$ is a $\mathcal{C}_n$-submodule of $M_{\leq 0}$.

\begin{lemma}\label{L:pi}
  The projection $\pi \colon M\rightarrow M_{\leq 0}/(\lambda(\mathcal{C}_{n}\otimes M_{>0})\cap M_{\leq 0})$ which maps $m$ to $0$ if $m\in M_{>0}$ and to its class if $m\in M_{\leq 0}$ is a morphism of $\mathcal C_n$-modules.
\end{lemma}

\begin{proof} For $m\in M$ and $x\in \mathcal C_n$, let us write $x\cdot m$ for $\lambda(x\otimes m)$. 

  Assume $m\in M_{>0}$.
  If $x\cdot m\in M_{>0}$, then $\pi(x\cdot m)=0=x\cdot \pi(m)$. If $x\cdot m\in M_{\leq 0}$, then $x\cdot m\in \lambda(\mathcal{C}_{n}\otimes M_{>0})\cap M_{\leq 0}$, hence $\pi(x\cdot m)=0=x\cdot\pi(m)$.

  Assume $m\in M_{\leq 0}$. Since $M_{\leq 0}$ is a $\mathcal C_n$-submodule of $M$,  $\pi(x\cdot m)=x\cdot \pi(m)$.
\end{proof}

\begin{prop}\label{P:leftadj}
The natural inclusion functor
$i\colon (\mathcal{C}_n\Mod)_{\text{-}\N,\Z}\longrightarrow (\mathcal{C}_n\Mod)_{\Z,\Z}$ has a left adjoint $t$ given on objects by
\[
  t(M)=M_{\leq 0}/(\lambda(\mathcal{C}_{n}\otimes M_{>0})\cap M_{\leq 0})~ \text{ for a $\mathcal{C}_n$-module $M$,}
\]
and on morphisms by sending a map of $\mathcal C_n$-modules to the induced map on the subquotient.
\end{prop}

\begin{proof}
Let $M\in (\mathcal{C}_n\Mod)_{\Z,\Z} $ and $N\in (\mathcal{C}_n\Mod)_{\text{-}\N,\Z}$.
Given a morphism $\tilde{f}\colon t(M)\longrightarrow N$ in $ (\mathcal{C}_n\Mod)_{\text{-}\N,\Z}$, consider the composite
\[
f\colon M\overset{\pi}{\longrightarrow}t(M)\overset{\tilde{f}}\longrightarrow i(N)=N,
\]
where $\pi$ is the morphism of $\mathcal C_n$-modules defined in Lemma \ref{L:pi}, so that  $f$ is a morphism of $\mathcal{C}_n$-modules. On the other hand, if $f\colon M\longrightarrow i(N)=N$ is a morphism of $\mathcal{C}_n$-modules, then $M_{>0}\subseteq \ker f$ and $\lambda(\mathcal{C}_{n}\otimes M_{>0})\cap M_{\leq 0}\subseteq \ker f$.
Hence, $f$ induces a morphism $\tilde{f}\colon t(M)\longrightarrow N$ such that $f=\tilde{f}\pi$.
\end{proof}

Theorem~\ref{modelE} shows that for each $r \geq 0$, there is a cofibrantly generated model structure on $\nCh_{\Z,\Z}$ where a map $f$ is a weak equivalence if it is an $E_r$-quasi-isomorphism, and a fibration if $E_i(f)$ is surjective
for $0\leq i\leq r$.
The generating cofibrations and generating trivial cofibrations are denoted $(I^n_r)'$ and $(J^n_r)'$ respectively. An application of the transfer theorem gives the following.

\begin{prop}\label{mczn}
  For each $r \geq 0$, there is a cofibrantly generated model structure on\/ $\nCh_{\text{-}\N,\Z}$, where 
  \begin{enumerate}
    \item weak equivalences are $E_r$-quasi-isomorphisms, 
    \item fibrations are morphisms of $n$-multicomplexes $f \colon A \to B$ such that $E_i(f)$ is  bidegree-wise surjective for every $0\leq i\leq r$, and
    \item the generating cofibrations and generating trivial cofibrations are $t(I^n_r)'$ and $t(J^n_r)'$ respectively.
  \end{enumerate}
\end{prop}

\begin{proof}
  We apply the transfer theorem to the adjunction $t \dashv i$ of Proposition~\ref{P:leftadj}.
The descriptions of the weak equivalences and fibrations are immediate as long as the transfer theorem holds. 
  We check the conditions (\ref{transfer-small}) and (\ref{transfer-functorial}) in the transfer theorem. Every $n$-multicomplex is $r$-fibrant, so the first part of (\ref{transfer-functorial}) trivially holds. 
  Condition (\ref{transfer-small})  holds as the functor $t$  preserves small objects. 
  It remains to find functorial path objects for $(-\N,\Z)$-graded $n$-multicomplexes. These exist because if  $A\in \nCh_{\text{-}\N,\Z}$, then $P_r(A)\in  \nCh_{\text{-}\N,\Z}$ by Remark~\ref{R:P_r}.
\end{proof}

It is also possible to transfer the model category structure to the upper half-plane in the case $r=0$. 
Similarly to above, we prove that the inclusion functor from $\nCh_{\Z,\N}$ to $\ncpx$ has a left adjoint.

\begin{prop}
For  $A \in \ncpx$, there is a $(\Z, \N)$-graded $n$-multicomplex given by
\[
  (t^{\prime}A)^{p,q} = 
  \begin{cases}
    A^{p,q} & q>0 \\
    A^{p,0}/d_0(A^{p,-1}) & q=0\\
    0 &q<0,
  \end{cases}
\]
with structure maps $d_i$ induced from those of $A$. Furthermore, this construction is functorial and
there exists an adjunction
\begin{center}
  \begin{tikzcd}
    \ncpx \arrow[r,shift left=.5ex,"t^{\prime}"] &
    \nCh_{\Z,\N} \arrow[l,shift left=.5ex,"i"]
  \end{tikzcd}
\end{center}
where $i$ is the natural inclusion functor and the functor $t^{\prime}$ is its left adjoint.

\end{prop}

\begin{proof}
We  check that for any $A\in \ncpx$, $t^{\prime}(A)$ is an  $n$-multicomplex.
Consider  $A$ as a $\mathcal{C}_n$-module $(A,d_0, \lambda)$  in a natural way (see  Proposition~\ref{P:Cn-modules}).
Let $A_{*,-1}$ and $A_{q<0}$ denote the following bigraded $R$-modules 
\[
A_{*,-1}^{p,q}=\begin{cases}
A^{p,-1}~&\text{if}~q=-1\\
0~&\text{otherwise}
\end{cases}~~\text{and}~~
A_{q<0 }^{p,q}=\begin{cases}
A^{p,q}~&\text{if}~q<0\\
0~&\text{otherwise}.
\end{cases}
\]
These are not  vertical bicomplexes in general, but $A_{q<0}\oplus d_0(A_{*,-1})$ is. Furthermore,  this
  is a $\mathcal{C}_n$-submodule of $A$ and the quotient  $A/(A_{q<0}\oplus d_0(A_{*,-1}))$ is a $\mathcal{C}_n$-module which corresponds to $t^{\prime}(A)$. Hence $t^{\prime}(A)$ is an $n$-multicomplex.

The functor $t^{\prime}$ is a left adjoint.
Let $\pi\colon A\longrightarrow t^{\prime}(A)$ be the projection  in $\ncpx$. For  $B\in \nCh_{\Z,\N}$,     a morphism $f \colon A \to i(B)$ in $\ncpx$ satisfies $d_0 f (A_{*,-1}) = f d_0 (A_{*,-1}) = 0$ and $f(A_{q<0})=0$. Hence $A_{q<0}\oplus d_0(A_{*,-1})$ is contained in $\ker f$ which implies that $f$ corresponds to a well defined morphism $\tilde{f}\colon t^{\prime}(A) \longrightarrow B$ such that  $ f= \tilde{f}\pi$.
\end{proof}

\begin{prop}\label{P:upperhalfplane}
  For  $r = 0$, there is a cofibrantly generated model structure on\/ $\nCh_{\Z,\N}$, where 
  \begin{enumerate}
    \item weak equivalences are $E_0$-quasi-isomorphisms, 
    \item fibrations are morphisms of $n$-multicomplexes $f \colon A \to B$ such that $f$ is bidegreewise surjective, and
    \item the generating cofibrations and generating trivial cofibrations are {$t^{\prime}I^n_0$ and $t^{\prime}J^n_0$ respectively}.
  \end{enumerate}
\end{prop}
\begin{proof}
 The proof proceeds in the same way as that of Proposition~\ref{mczn}, using the existence of a functorial path object $P_0(A)$ for the category $\nCh_{\Z,\N}$ when $r=0$ (see Remark~\ref{R:P_r}).
\end{proof}

\section{Examples of cofibrancy and cofibrant replacement}
\label{S:cofibrants}

In this section we give some examples of cofibrant and non-cofibrant objects.
We will see that all the objects appearing in our generating (trivial) cofibrations for the model structures of 
Theorem~\ref{modelE} have trivial total homology. This leads naturally to the question of how one can build cofibrant objects with non-trivial total homology and we  explore this here. In particular, we note that the ground ring $\kk$ concentrated in a single bidegree is not a cofibrant object and we describe a cofibrant replacement in $n$-multicomplexes. For example, in the case of bicomplexes, this is an ``infinite staircase''. 
We also consider briefly what happens under transfer of model structures to bounded versions.

\begin{exmp}\label{ex:noncofib}
For any $p,q \in \Z$, the $n$-multicomplex
 $ \kk\bgd{p}{q}$ is not cofibrant in $(\ncpx)_r$ for $2\leq n\leq \infty$ and $r\geq 0$.
  Consider the ``corner'' bicomplex, $C(p,q)$, pictured below.
  \[\xymatrix{
      \kk\bgd{p-1}{q}&\ar[l]_{1}\kk\bgd{p}{q}\\
      \kk\bgd{p-1}{q-1}\ar[u]^{1}&\\
  }\]
  We can view this as an $n$-multicomplex for $2\leq n\leq\infty$.
  Define the map of $n$-multicomplexes $\pi \colon C(p,q) \to \kk\bgd{p}{q}$ to be the identity on $\kk$ in bidegree $(p,q)$ and zero in
  all other bidegrees. Then $\pi$  is clearly bidegreewise surjective, so a $0$-fibration. Also, $E_1(C(p,q))= \kk\bgd{p}{q}$
  and $E_1(\pi)$ is the identity map of $ \kk\bgd{p}{q}$. Thus $\pi$ is a trivial $0$-fibration.

  Now we can test against this trivial $0$-fibration to see that $ \kk\bgd{p}{q}$ is not $0$-cofibrant.
  Indeed we find that there is no lift 
  $ \kk\bgd{p}{q}\to C(p,q)$ in the diagram  of $n$-multicomplexes
  \[\xymatrix{
      &C(p,q)\ar[d]_{\pi}\\
      \kk\bgd{p}{q}\ar[r]_-{1}\ar@{.>}[ur]^{\nexists}& \kk\bgd{p}{q}
  }\]

  Any such lift $f$ would have to take the generator $\mathbb{1}_R$ to the generator $\mathbb{1}_R$ in bidegree $(p,q)$ in $C(p,q)$, but
  then for $f$ to be a map of bicomplexes it would have to satisfy $0=f(d_1 \mathbb{1}_R)=d_1f(\mathbb{1}_R)=\mathbb{1}_R$, giving a contradiction. 
  
  Since $ \kk\bgd{p}{q}$ is not $0$-cofibrant, it is not $r$-cofibrant  for any $r$.
\end{exmp}

\begin{prop}\label{cofZW}
For $p,q \in \Z$, $r, s\geq 0$ and $2\leq n\leq\infty$, $\Zzw^n_{s}(p,q)$ is cofibrant in $(\ncpx)_r$.
\end{prop}

\begin{proof}
Fix  $p,q \in \Z$, $r, s\geq 0$ and $n$ with $2\leq n\leq \infty$. Note that a lift exists in the diagram of $n$-multicomplexes
  \[\xymatrix{
      &A\ar[d]_{f}\\
    \Zzw^n_{s}(p,q)\ar[r]\ar@{.>}[ur]& B
  }\]
if and only if $ZW_s(f)$ is surjective in bidegree $(p,q)$. Now suppose that $f$ is an $r$-trivial $r$-fibration. Then $E_i(f)$ is surjective for all 
$i\geq 0$. Using Remark~\ref{mainobs}, it follows that  $ZW_s(f)$ is surjective for all $s$. So the required lift exists.
\end{proof}

\begin{rem}
If we use the $r$-model structure of Theorem~\ref{modelZ} instead, the same line of argument shows
that $\Zzw^n_{s}(p,q)$ is $r$-cofibrant for $s\geq r$.
\end{rem}

\begin{corollary} For every $p,q \in \Z$, $s\geq 0$  and   $2\leq n\leq \infty$, we have
\[E_i(\Zzw^n_{s}(p,q))=\begin{cases}
\kk^{p,q}\oplus\kk^{p-s,q-s+1}& \text{if } 1\leq i\leq s\\
0& \text{if } i\geq s+1.\end{cases}\]
\end{corollary}

\begin{proof} The $n$-multicomplex $\Zzw^n_{s}(p,q)$ is $r$-cofibrant for any $r\geq 0$ by 
Proposition~\ref{cofZW}. We claim that $p_{n,2}(\Zzw^n_{s}(p,q))=\Zzw^2_{s}(p,q)$ {(see Section \ref{S:quillenequivalence} for the definition of $p_{n,2}$)}. This follows from the
definition of $\Zzw^n_{s}(p,q)$ via successive pushout (Definition~\ref{D:curlyZW}) and the fact that
 $p_{n,2}$ is a left adjoint and so preserves pushouts, together with the initial cases
 $p_{n,2}(\Zzw^n_{0}(p,q))=\Zzw^2_{0}(p,q)$ and
 $p_{n,2}(d_0^*)=d_0^*$. By Theorem~\ref{Qequiv}, since $\Zzw^n_{s}(p,q)$ is $r$-cofibrant,
 the unit of the adjunction $\Zzw^n_{s}(p,q)\rightarrow \Zzw^2_{s}(p,q)$ is an $E_r$-quasi-isomorphism, for each $r\geq 0$, in particular an $E_0$-quasi-isomorphism. For the staircase bicomplex
 $\Zzw^2_{s}(p,q)$ it is easy to read off the pages of the spectral sequence directly:
\[E_i(\Zzw^2_{s}(p,q))=\begin{cases}
\kk^{p,q}\oplus\kk^{p-s,q-s+1}& \text{ if } 1\leq i\leq s\\
0& \text{ if } i\geq s+1,\end{cases}\]
as required.
\end{proof}

\begin{defn}\label{ZWinfinity}
Let $p,q \in \Z$ and $2\leq n\leq \infty$.
We define $\Zzw_\infty^n(p,q)=\varinjlim_s \Zzw_s^n(p,q)$, where the 
 colimit is taken over the  maps $\Zzw_s^n(p,q)\to \Zzw_{s+1}^n(p,q)$ representing the
projection maps $ZW_{s+1}^n\to ZW_s^n$.
\end{defn}

\begin{exmp}
When $n=2$, the map $\Zzw_s^2(p,q)\to \Zzw_{s+1}^2(p,q)$ is the inclusion of a staircase with $s$-horizontal steps into a staircase with $s+1$-horizontal steps
and $\Zzw_\infty^2(p,q)$ is the infinite (downwards to the left) staircase bicomplex, with top right entry in bidegree $(p,q)$:
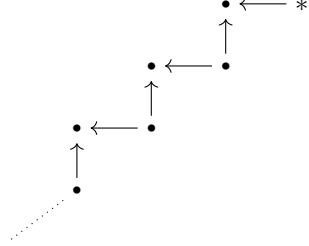
\[
  \begin{tikzcd}[cramped,sep=scriptsize]
  &                                              &                            & \kkpic                     & \kkpqpic \arrow[l] \\
  &                                              & \kkpic                     & \kkpic \arrow[u] \arrow[l] &                    \\
  & \kkpic                                       & \kkpic \arrow[u] \arrow[l] &                            &                    \\
  & \kkpic \arrow[u] \arrow[ld, no head, dotted] &                            &                            &                    \\
  \phantom{.} &                                              &                            &                            &                   
\end{tikzcd}
\]
\end{exmp}

\begin{prop}
Let $p,q \in \Z$ and $2\leq n\leq \infty$. Then 
$\Zzw_\infty^n(p,q)\to \kk^{p,q}$ given by projection to $\kk^{p,q}$ is an $r$-cofibrant replacement of $\kk^{p,q}$
for all $r\geq 0$.
\end{prop}

\begin{proof}
First we check that $\Zzw_\infty^n(p,q)$ is $r$-cofibrant for all $r\geq 0$. 
The relevant lift exists for $\Zzw_\infty^n(p,q)$ if and only if compatible lifts exist for each $\Zzw_s^n(p,q)$. Such lifts
do exist for each $\Zzw_s^n(p,q)$ by Proposition~\ref{cofZW} and it is straightforward to check that they are compatible.

The map $E_r(\Zzw_s^n(p,q)\to \Zzw_{s+1}^n(p,q))$ is the projection to $\kk^{p,q}$ if $1\leq r\leq s$ and $0$ otherwise, so we 
see that  $E_r(\Zzw^n_{\infty}(p,q))=\kk^{p,q}$ for all $r\geq 1$.  And the projection $\Zzw_\infty^n(p,q)\to \kk^{p,q}$ induces
an isomorphism on $E_r$ for all $r\geq 1$, that is, it is an $E_r$-quasi-isomorphism for all $r\geq 0$.
\end{proof}

\subsection{Upper half-plane versions}

We consider the $r=0$ model structure on upper half-plane $n$-multicomplexes from Proposition~\ref{P:upperhalfplane}.
The generating cofibrations and generating trivial cofibrations are given by $t'I_0$ and $t'J_0$. The interesting new thing that appears is the cotruncation of $\iota_{1} \colon \Zzw_{1}(p,0)\to\Bbw_{1}(p,-1)$, which is
$t'\iota_1=0 \colon \Zzw_{1}(p,0)\to 0$.
This allows one to see that $\kk^{p,0}$ is $0$-cofibrant, since we have a pushout diagram
\[\xymatrix{
    \Zzw_1(p-1,0)\ar[d]\ar[r] & 0\ar[d]\\
    \Zzw_1(p,0)\ar[r] & \kk^{p,0}
}\]
{where the top horizontal map is a cofibration and $\Zzw_1(p,0)$ is cofibrant.}
On the other hand, $\kk^{p,q}$ for $q>0$ is not $0$-cofibrant, just as in Example~\ref{ex:noncofib}.
This shows (unsurprisingly) that in the $0$-model structure on upper half-plane $n$-multicomplexes, cofibrancy is not preserved under vertical shift.

\bibliographystyle{abbrvurl}

\setlength{\parindent}{0pt}

\end{document}